\documentclass[10pt,fullpage]{article}
\usepackage{amsmath}
\usepackage{amsfonts,amssymb}
\usepackage{amsthm}
\usepackage{hyperref}
\usepackage[right = 2.5cm, left=2.5cm, top = 2.5cm, bottom =2.5cm]{geometry}
\newcommand{\Real}{\mathbb R}
\newcommand{\norm}[1]{\|#1\|}
\newcommand{\abs}[1]{\left\vert#1\right\vert}
\newcommand{\set}[1]{\left\{#1\right\}}

\newcommand{\grad}{\nabla}

\newcommand{\mollify}[1]{ \mathcal{J}_\epsilon #1 }
\newcommand{\mollifyj}[1]{\mathcal{J}_{\epsilon_j} #1 }
\newcommand{\conv}[2]{#1 \ast #2}
\newcommand{\D}{D}
\newcommand{\K}{\mathcal{K}}
\newcommand{\ineqtext}[1]{ ^{\text{\tiny #1}}}
\newcommand{\wknorm}[2]{\norm{#1}_{L^{#2,\infty}}}
\newcommand{\wkspace}[1]{L^{#1,\infty}}
\newcommand{\F}{\mathcal{F}}

\newcommand{\reg}[1]{#1^\epsilon}

\newcommand{\energy}{\mathcal{F}}

\newcommand{\kernel}{\mathcal{K}}

\newcommand{\intt}{\int_{D_T}}
\newcommand{\ball}{B_n}

\newcommand{\pd}{D_T}
\newcommand{\pdRd}{\Real^d_T}

\newtheorem{theorem}{Theorem}
\theoremstyle{remark}
\newtheorem{remark}{Remark}

\theoremstyle{theorem}

\newtheorem{proposition}{Proposition}
\theoremstyle{definition}
\newtheorem{definition}{Definition}
\theoremstyle{lemma}
\newtheorem{lemma}{Lemma}

\begin{document}

\title{Local and Global Well-Posedness for Aggregation Equations and Patlak-Keller-Segel Models with Degenerate Diffusion}

\author{Jacob Bedrossian \footnote{\textit{jacob.bedrossian@math.ucla.edu}, University of California-Los Angeles, Department of Mathematics} and Nancy Rodr\'iguez\footnote{\textit{nrodriguez@math.ucla.edu}, University of California-Los Angeles, Department of Mathematics} and Andrea L. Bertozzi\footnote{\textit{bertozzi@math.ucla.edu}, University of California-Los Angeles, Department of Mathematics}}

\maketitle

\begin{abstract}
Recently, there has been a wide interest in the study of aggregation equations and Patlak-Keller-Segel (PKS) models for chemotaxis with degenerate diffusion.  The focus of this paper is the unification and generalization of the well-posedness theory of these models.  
We prove local well-posedness on bounded domains for dimensions $d\geq 2$ and in all of space for $d\geq 3$, the uniqueness being a result previously not known for PKS with degenerate diffusion. We generalize the notion of criticality for PKS and show that subcritical problems are globally well-posed.  For a fairly general class of problems, we prove the existence of a critical mass which sharply 
divides the possibility of finite time blow up and global existence. Moreover, we compute the critical mass for fully general problems and show that solutions with smaller mass exists globally. For a class of supercritical problems we prove finite time blow up is possible for initial data of arbitrary mass.  
\end{abstract}

\section{Introduction}
Nonlocal aggregation phenomena have been studied in a wide variety of biological applications such as migration patterns in ecological systems \cite{Bio,Topaz,Milewski,GurtinMcCamy77,Burger07} and
Patlak-Keller-Segel (PKS) models of chemotaxis \cite{Filbert,Patlak,Herrero2,KS,Lapidus}.  
Diffusion is generally included in these models to account for the dispersal of organisms.
Classically, linear diffusion is used, however recently, there has been a widening interest in models with degenerate diffusion to include over-crowding effects \cite{Topaz,Burger07}. 
The parabolic-elliptic PKS is the most widely studied model for aggregation, where the nonlocal effects are modeled by convolution with the Newtonian or Bessel potential. On the other hand, in population dynamics, the nonlocal effects are generally modeled with smooth, fast-decaying kernels. However, all of these models are describing 
the same mathematical phenomenon: the competition between nonlocal aggregation and diffusion. 
For this reason, we are interested in unifying and extending the local and global well-posedeness theory of general aggregation models with degenerate diffusion of the form
\begin{subequations}\label{AE}
\begin{align}
u_t +  \grad \cdot \left(u\vec{v} \right) = \Delta A(u) &\quad \text{in}\; [0,T)\times D, \label{PD}\\
\vec{v} = \conv{\grad \K}{u} \label{CG}. 
\end{align}
\end{subequations}
Mathematical works most relevant this paper are those with degenerate diffusion \cite{BertozziSlepcev10,SugiyamaDIE06,SugiyamaADE07,SugiyamaDIE07,Blanchet09,KowalczykSzymanska08,Kowalczyk05,CalvezCarrillo06} and those from the classical PKS literature \cite{JagerLuckhaus92,Dolbeault04,BlanchetEJDE06,Blanchet08}. See also \cite{KarchSuzuki10}. \\

\noindent
Existence theory is complicated by the presence of degenerate diffusion and singular kernels such as the Newtonian potential.
Bertozzi and Slep\v{c}ev in \cite{BertozziSlepcev10} prove existence and uniqueness of models with general diffusion but restrict to non-singular kernels.
Sugiyama \cite{SugiyamaDIE07} proved local existence for models with power-law diffusion and the Bessel potential for the kernel, but uniqueness of solutions was left open. 
We extend the work of \cite{BertozziSlepcev10} to prove the local existence of \eqref{AE} with degenerate diffusion and singular kernels including the Bessel and Newtonian potentials.
The existing work on uniqueness of these problems included a priori regularity assumptions \cite{KowalczykSzymanska08} or the use of entropy solutions \cite{Burger07} (see also \cite{Carrillo99}). 
The Lagrangian method introduced by Loeper in \cite{LoeperVP06} estimates the difference of weak solutions in the Wasserstein distance and is very useful for inviscid problems or problems with linear diffusion \cite{LoeperSG06,BertozziLaurentRosado10,CarrilloRosado10}. 
In the presence of nonlinear diffusion, it seems more natural to approach uniqueness in $H^{-1}$, where the diffusion is monotone (see \cite{VazquezPME}). This is the approach taken in \cite{BertozziBrandman10,BertozziSlepcev10}, which we extend to handle singular kernels such as the Newtonian potential, proving uniqueness of weak solutions with no additional assumptions, provided the domain is bounded 
or $d \geq 3$.
The main difference is the use of more refined estimates to handle the lower regularity of $\grad K \ast u$, similar to the traditional proof of uniqueness of $L^1 \cap L^\infty$-vorticity solutions to the 2D Euler equations \cite{Yudovich63,MajdaBertozzi} and a similar proof of the uniqueness of $L^1 \cap L^\infty$ solutions to the Vlasov-Poisson equation \cite{Robert97}. \\

\noindent
There is a natural notion of criticality associated with this problem, which roughly corresponds to the balance between the aggregation and diffusion. 
For problems with homogeneous kernels and power-law diffusion, $\K = c\abs{x}^{2-d}$ and $A(u) = u^m$, a simple scaling heuristic suggests that these forces are in balance if $m = 2-2/d$ \cite{Blanchet09}.
If $m > 2-2/d$ then the problem is subcritical and the diffusion is dominant. On the other hand, if $m < 2-2/d$ then the problem is supercritical and the aggregation is dominant. For the PKS with power-law diffusion, Sugiyama showed global existence for subcritical problems and that finite
time blow up is possible for supercritical problems \cite{SugiyamaDIE07,SugiyamaADE07,SugiyamaDIE06}. We extend this notion of criticality to general problems by observing that only the behavior of the solution at high concentrations will divide finite time blow up from global existence (see Definition \ref{def:criticality}). We show global well-posedness for subcritical problems and finite time blow up for certain supercritical problems. \\

\noindent
If the problem is critical, it is well-known that in PKS
there exists a critical mass, and solutions with larger mass can blow up in finite time \cite{BlanchetEJDE06,JagerLuckhaus92,Biler06,Dolbeault04,Blanchet08,Calvez,Blanchet09,SugiyamaDIE06,SugiyamaADE07,CalvezCarrillo06}. For linear diffusion, the same critical mass has been identified for the Bessel and Newtonian potentials \cite{BlanchetEJDE06,Calvez};
however for nonlinear diffusion, the critical mass has only been identified for the Newtonian potential \cite{Blanchet09}. 
In this paper we extend the free energy methods of \cite{Blanchet09,Dolbeault04,CalvezCarrillo06,Blanchet08} to estimate the critical mass for a wide range of kernels and nonlinear diffusion, which include these known results. For a smaller class of problems, 
including standard PKS models, we show this estimate is sharp. \\

\noindent
The problem \eqref{AE} is formally a gradient flow with respect to the 
Euclidean Wasserstein distance for the \emph{free energy} 
\begin{align}
\F(u(t)) & = S(u(t)) - \mathcal{W}(u(t)), \label{def:FreeEnergy}
\end{align}
where the \emph{entropy} $S(u(t))$ and the \emph{interaction energy} $\mathcal{W}(u(t))$ are given by
\begin{align*}
S(u(t)) & = \int \Phi(u(x,t)) dx, \\ 
\mathcal{W}(u(t)) & = \frac{1}{2}\int\int u(x,t)\K(x-y)u(y,t) dx dy.
\end{align*}
For the degenerate parabolic problems we consider, the \emph{entropy density} $\Phi(z)$ is a strictly convex function satisfying
\begin{equation}
\Phi^{\prime\prime}(z) = \frac{A^\prime(z)}{z}, \;\;\; \Phi^\prime(1) = 0, \;\;\; \Phi(0) = 0. \label{def:G}
\end{equation}
See \cite{CarrilloEntDiss01} for more information on these kinds of entropies. 
Although there is a rich theory for gradient flows of this general type when the 
kernel is regular and $\lambda$-convex \cite{McCann97,AmbrosioGigliSavare,CarrilloDiFrancesco09} 
the kernels we consider here are more singular and the notion of displacement convexity introduced in \cite{McCann97} no longer holds.
For this reason, the rigorous results of the gradient flow theory are not directly applicable, however, certain aspects may be recovered, such as the use of steepest descent schemes \cite{BlanchetCalvezCarrillo08,BlanchetCarlenCarrillo10}.  
Moreover, the free energy \eqref{def:FreeEnergy} is still the important dissipated quantity in the global existence and finite time blow up arguments.
The free energy has been used by many authors for the same purpose, see for instance \cite{SugiyamaDIE06,BlanchetEJDE06,CalvezCarrillo06,Blanchet09,BertozziLaurent07,Blanchet08}. For the remainder of the paper we only consider initial data with finite free energy, although the local existence arguments
may hold in more generality. \\ 

\noindent
There is a vast literature of related works on models similar to \eqref{AE}. For literature on PKS we refer the reader to the review articles \cite{Hortsmann,HandP}; see also \cite{HandP2,Rosado,Calvez} for parabolic-parabolic Keller-Segel systems.
 For the inviscid problem, see the recent works of \cite{Laurent07,BertozziLaurent07,BertozziBrandman10,BertozziLaurentRosado10,CarrilloDiFrancesco09}. For a study of these equations with fractional linear diffusion see \cite{LiRodrigo09,LiRodrigoAM09,Biler09}.
When the diffusion is sufficiently nonlinear and the kernel is in $L^1$, \eqref{AE} may be written as a regularized interface problem, a notion studied in \cite{Slepcev08}. Critical mass behavior is also a property of other related critical PDE, such as the marginal unstable thin film equation \cite{Witelski04,BertozziPugh98} and critical nonlinear Schr\"odinger equations
\cite{Weinstein83, KillipVisanClay}. \\

\noindent
{\it Outline of Paper}.  In Section \S\ref{sec:DefnAssump} we state the relevant definitions and notation.  Furthermore, we give a summary of the main results but reserve the proofs for subsequent sections.  In Section \S\ref{sec:Unique} we prove the uniqueness result.  Local existence is proved in Section \S\ref{sec:LocExist}.  The first result is proved on bounded domains in $d\geq 2$ and the second is proved on all space for $d\geq 3$.  In Section \S\ref{sec:ContThm} we prove a continuation theorem.  The global existence results are proved in Section \S\ref{sec:GE}.
Finally, in Section \S\ref{sec:FTBU} we prove the finite time blow up results. 

\subsection{Definitions and Assumptions} \label{sec:DefnAssump}
We consider either $\D = \Real^d$ with $d \geq 3$ or $\D$ smooth, bounded and convex with $d \geq 2$, in which case we impose no-flux conditions
\begin{equation}
(-\nabla A(u) + u \nabla \mathcal{K}\ast u) \cdot \nu = 0 \text{ on } \partial D \times [0,T), \label{cond:no_flux}
\end{equation}
where $\nu$ is the outward unit normal to $\D$. We neglect the case $D = \Real^2$ for technicalities introduced by 
the logarithmic potential. \\

\noindent
We denote $D_T := (0,T) \times D$. We also denote $\norm{u}_{p} := \norm{u}_{L^p(D)}$ where $L^p$ is the standard Lebesgue space. 
We denote the set $\set{u > k} := \set{x\in D: u(x) > k}$, if $S \subset \Real^d$ then $\abs{S}$ denotes the Lebesgue measure and $\mathbf{1}_{S}$ denotes the standard characteristic function. In addition, we use $\int f dx := \int_D f dx$, and only indicate the domain of integration 
where it differs from $D$. 
We also denote the weak $L^p$ space by $L^{p,\infty}$ and the associated quasi-norm 
\begin{equation*}
\wknorm{f}{p} = \left( \sup_{\alpha > 0} \alpha^p \lambda_f(\alpha) \right)^{1/p},
\end{equation*}
where $\lambda_f(\alpha) = \abs{\set{f > \alpha}}$ is the distribution function of $f$. 
Given an initial condition $u(x, 0)$ we denote its mass by $\int u(x, 0) dx= M$. In formulas we use the notation $C(p,k,M,..)$ to denote a generic constant, which may be different from line to line or even term to term in the same computation. 
In general, these constants will depend on more parameters than those listed, for instance those associated with the problem such as $\K$ and the dimension but these
dependencies are suppressed.
We use the notation $f \lesssim_{p,k,...} g$ to denote $f \leq C(p,k,..)g$ where again, dependencies that are not relevant are suppressed. 
 \\

\noindent
We now make reasonable assumptions on the kernel which include important cases of interest, such as when $\K$ is the fundamental solution of an elliptic PDE. 
To this end we state the following definition. 

\begin{definition}[Admissible Kernel] \label{def:admK}
We say a kernel $\K$ is \emph{admissible} if $\K \in W^{1,1}_{loc}$ and the following holds:
\begin{itemize}
\item[\textbf{(R)}] $\K \in C^3\setminus\set{0}$.
\item[\textbf{(KN)}] $\K$ is radially symmetric, $\K(x) = k(\abs{x})$ and $k(\abs{x})$ is non-increasing.
\item[\textbf{(MN)}] $k^{\prime\prime}(r)$ and $k^\prime(r)/r$ are monotone on $r \in (0,\delta)$ for some $\delta > 0$. 
\item[\textbf{(BD)}] $\abs{D^3\K(x)} \lesssim \abs{x}^{-d-1}$. 
\end{itemize}
\end{definition}
\noindent
This definition ensures that the kernels we consider are radially symmetric, non-repulsive, reasonably well-behaved at the origin, and have second derivatives which define bounded distributions on $L^p$ for $1 < p < \infty$ (see Section \S\ref{sec:PropAdmKern}). 
These conditions imply that if $\kernel$ is singular, the singularity is restricted to the origin.  Note also, that the Newtonian and Bessel potentials are both admissible for all dimensions $d \geq 2$; hence, the PKS and related models are included in our analysis. \\

\noindent
We now make precise what kind of nonlinear diffusion we are considering. 

\begin{definition}[Admissible Diffusion Functions] \label{def:admDiff}
We say that the function $A(u)$ is an admissible diffusion function if: 
\begin{itemize}
\item[\textbf{(D1)}] $A\in C^1([0,\infty))\;\text{with}\;A'(z)>0\;\text{for}\;z\in (0,\infty)$.
\item[\textbf{(D2)}] $A'(z)>c\;\text{for}\;z > z_c$ for some $c,z_c > 0$. 
\item[\textbf{(D3)}] $\int_0^1 A^\prime(z)z^{-1} dz < \infty$. 
\end{itemize}
\end{definition}

\noindent
This definition includes power-law diffusion $A(u) = u^m$ for $m > 1$. Note that \textbf{(D3)} requires the diffusion to be degenerate at $u = 0$, however it is permitted to behave linearly at infinity. 
Furthermore, on bounded domains condition {\bf (D3)} can be relaxed without any significant modification to the methods.  Following \cite{BertozziSlepcev10}, the notions of weak solution are defined separately for bounded and unbounded domains.

\begin{definition}[Weak Solutions on Bounded Domains] \label{def:WSBD}
Let $A(u)$ and $\kernel$ be admissible, and $u_0(x)\in L^\infty(D)$ be non-negative.  A non-negative function $u:[0,T] \times D \rightarrow [0,\infty)$ is a weak solution to \eqref{AE} if $u\in L^\infty(D_T)$, $A(u) \in L^2(0,T,H^1(D))$, $u_t \in L^2(0,T,H^{-1}(D))$ and 
\begin{align}\label{WF2}
\int_0^T\int u\phi_t\; dxdt = \int u_0(x)\phi(0,x)dx + \int_0^T\int (\nabla A(u) - u\nabla \kernel \ast u) \cdot \nabla \phi\; dxdt,
\end{align}
for all $\phi\in C^{\infty}(\overline{D_T})$ such that $\phi(T)=0$.   
\end{definition}
\noindent It follows that $u\grad \K \ast u \in L^2(D_T)$; therefore, definition \ref{def:WSBD} is equivalent to the following,
\begin{align}\label{WF1}
\left\langle u_t(t), \phi\right\rangle = \int \left(-\nabla A(u) + u\nabla \kernel \ast u\right)\cdot \nabla \phi\; dx,
\end{align}
\noindent  for all test functions $\phi \in H^1$ for almost all $t\in [0,T]$.  Above $\left\langle\cdot ,\cdot \right\rangle$ denotes the standard dual pairing between $H^1$ and $H^{-1}$.   Similarly for $\Real^d$ we define the following notion of weak solution as in \cite{BertozziSlepcev10}.

\begin{definition}[Weak Solution in $\Real^d$, $d\geq 3$] \label{def:WSRD}
Let $A$ and $\kernel$ be admissible, and $u_0\in L^\infty(\Real^d)\cap L^1(\Real^d)$ be non-negative.  A function $u:[0,T] \times \Real^d \rightarrow [0,\infty)$ is a weak solution of \eqref{AE} if $u\in L^\infty((0,T)\times\Real^d)\cap L^\infty(0,T,L^1(\Real^d))$, $A(u)\in L^2(0,T,\dot{H}^1(\Real^d))$, $u\conv{\grad K}{u} \in L^2(D_T)$, $u_t\in L^2(0,T,\dot{H}^{-1}(\Real^d))$, and for all test functions $\phi \in \dot{H}^{1}(\Real^d)$ for a.e $t\in [0,T]$ \eqref{WF1} holds.  
\end{definition}

\noindent
We show below (Theorem \ref{thm:Uniqueness}) that weak solutions satisfying Definition \ref{def:WSBD} or \ref{def:WSRD} are in fact unique. Moreover, we show the unique weak solution satisfies the energy dissipation inequality (Proposition \ref{prop:EnrDiss}),
\begin{align}\label{EnrDiss}
\energy(u(t))+\int_0^t\int \frac{1}{u}\left|A'(u)\nabla u -u\nabla\kernel\ast u\right|^2dxdt \leq \energy(u_0(x)).  
\end{align} 
As \eqref{EnrDiss} is important for the global theory, one could also refer to these solutions as free energy solutions, as is done in \cite{Blanchet09}. Uniqueness implies that there is no distinction between free energy solutions in \cite{Blanchet09} and weak solutions. \\

\noindent
Since \eqref{AE} conserves mass, the natural notion of criticality is with respect to the usual mass invariant scaling $u_\lambda(x) = \lambda^du(\lambda x)$. A simple heuristic for understanding how this scaling plays a role in the global existence is seen by examining the case of power-law diffusion and homogeneous kernel, $A(u) = u^m$ and $\K(x) = \abs{x}^{-d/p}$. Under this mass invariant scaling the free energy \eqref{def:FreeEnergy} becomes,
\begin{equation*}
\F(u_\lambda) = \lambda^{dm-d}S(u) - \lambda^{d/p}\mathcal{W}(u). 
\end{equation*}
As $\lambda \rightarrow \infty$, the entropy and the interaction energy are comparable if $m = (p+1)/p$. We should expect global existence if $m > (p+1)/p$, as the diffusion will dominate as $u$ grows, and possibly finite time blow up if $m < (p+1)/p$ as the aggregation will instead be increasingly dominant.
We consider inhomogeneous kernels and general diffusion, however for the problem of global existence, only the behavior as $u \rightarrow \infty$ will be important, in contrast to the problem of local existence. 
Noting that $\abs{x}^{-d/p}$ is, in some sense, the representative singular kernel in $L^{p,\infty}$ leads to the following definition. 
This critical exponent also appears indirectly in \cite{LionsCC84}. 
\begin{definition}[Critical Exponent]\label{def:mstar}
Let $d\geq 3$ and $\K$ be admissible such that $\K \in \wkspace{p}_{loc}$ for some $d/(d-2) \leq p < \infty$. Then
the \emph{critical exponent} associated to $\K$ is given by
\begin{equation*}
1 < m^\star = \frac{p+1}{p} \leq 2 - 2/d.
\end{equation*}
If $D^2\K(x) = \mathcal{O}(\abs{x}^{-2})$ as $x \rightarrow 0$, then we take $m^\star = 1$.  
\end{definition}

\begin{remark}
The case $m^\star = 1$ implies at worst a logarithmic singularity as $x \rightarrow 0$ and if $d = 2$ then all admissible kernels have $m^\star = 1$. 
\end{remark}

\noindent
Now we define the notion of criticality. It is easier to define this notion in terms of the quantity
$A^\prime(z)$, as opposed to using $\Phi(z)$ directly. 
\begin{definition}[Criticality] \label{def:criticality}
We say that the problem is \emph{subcritical} if
\begin{equation*}
\liminf_{z \rightarrow \infty} \frac{A^\prime(z)}{z^{m^{\star}-1}} = \infty,
\end{equation*}
\emph{critical} if
\begin{equation*}
0 < \liminf_{z \rightarrow \infty} \frac{A^\prime(z)}{z^{m^{\star}-1}} < \infty,
\end{equation*}
and \emph{supercritical} if
\begin{equation*}
\liminf_{z \rightarrow \infty} \frac{A^\prime(z)}{z^{m^{\star}-1}} = 0.
\end{equation*}
\end{definition}

\noindent
Notice that in the case of power-law diffusion, $A(u) = u^m$, subcritical, critical and supercritical respectively correspond to $m > m^\star, m = m^\star$ and $m < m^\star$. Moreover, in the case of the Newtonian or Bessel potential, $m^\star = 2-2/d$ and the critical diffusion exponent of the PKS models discussed in \cite{SugiyamaADE07,SugiyamaDIE06,Blanchet09} is
recovered.

\subsection{Summary of Results}
\noindent
The proof of local existence follows the work of Bertozzi and Slep\v{c}ev \cite{BertozziSlepcev10}, where \eqref{AE} 
is approximated by a family of uniformly parabolic problems. 
The primary new difficulty, due to the singularity of the kernel, is obtaining uniform a priori $L^\infty$ bounds, which is 
overcome here using the Alikakos iteration \cite{Alikakos}. Solutions are first constructed on bounded domains. 

\begin{theorem}[Local Existence on Bounded Domains, $d \geq 2$]\label{LEADD}
Let $A(u)$ and $\kernel(x)$ be admissible.  Let $u_0(x)\in L^\infty(D)$ be a non-negative initial condition, then \eqref{AE} has a weak solution $u$ on $[0,T]\times D$, for some $T>0$. Additionally, $u\in C([0,T];L^p(D))$ for $p\in [1,\infty)$.
\end{theorem}

\noindent
In dimensions $d \geq 3$ we also construct local solutions on $\Real^d$ by taking the limit of solutions on bounded domains. 
\begin{theorem}[Local Existence in $\Real^d$, $d \geq 3$]\label{LEADDRd}
Let $A(u)$ and $\kernel(x)$ be admissible.  Let $u_0(x)\in L^\infty(\Real^d)\cap L^1(\Real^d)$ be a non-negative initial condition, then \eqref{AE} has a weak solution $u$ on $\Real^d_T$, for some $T>0$.  Additionally, $u\in C([0,T];L^p(\Real^d))$ for all $1\leq p< \infty$ and the mass is conserved.
\end{theorem}

\noindent
As previously mentioned, the free energy is a dissipated quantity for weak solutions and is a key tool for the global theory.

\begin{proposition}[Energy Dissipation] \label{prop:EnrDiss}
Weak solutions to \eqref{AE} satisfy the energy dissipation inequality \eqref{EnrDiss} for almost all $t\geq 0$.   
\end{proposition}

\noindent
As in \cite{BertozziSlepcev10}, uniqueness holds on bounded, convex domains in $d \geq 2$ or on $\Real^d$ for $d \geq 3$. The proof also holds for more general diffusion (e.g. fast or strongly degenerate diffusion) or no diffusion at all.

\begin{theorem}[Uniqueness] \label{thm:Uniqueness}
Let $D \subset \Real^d$ for $d \geq 2$ be bounded and convex, then weak solutions to \eqref{AE} are unique. The conclusion also holds on $\Real^d$ for $d \geq 3$. 
\end{theorem}

\noindent
We also prove the following continuation theorem, which generalizes similar theorems used in for instance \cite{BlanchetEJDE06,Blanchet09}.  
The proof follows the well known approach of first bounding intermediate $L^p$ norms and using Alikakos iteration \cite{Alikakos} to conclude the solution is bounded in $L^\infty$ (for instance, see \cite{Kowalczyk05,SugiyamaDIE07,Blanchet09,Corrias04,JagerLuckhaus92,BlanchetEJDE06}). 

\begin{theorem}[Continuation] \label{thm:Continuation}
The weak solution to \eqref{AE} has a maximal time interval of existence $T_\star$ and either $T_\star = \infty$ or $T_\star < \infty$ and 
\begin{equation}\label{cond:equint}
\lim_{k \rightarrow \infty}\limsup_{t \nearrow T_\star} \norm{(u-k)_+}_{\frac{2-m}{2-m^\star}} > 0.  
\end{equation}

Here $m$ is such that $1 \leq m \leq m^\star$ and $\liminf_{z \rightarrow \infty}A^{\prime}(z)z^{1-m} > 0$. 
In particular, for all $p > (2-m)/(2-m^\star)$, 
\begin{equation*}
\lim_{t \nearrow T_\star} \norm{u}_{p} = \infty.
\end{equation*}
\end{theorem}

\begin{remark}
Note that the order of the limits in Theorem \ref{thm:Continuation} is important.  In fact, if the ordered is reversed the limit is always zero.  
\end{remark}

\noindent
For the case $m^\star = 2-2/d$, Blanchet et al. \cite{Blanchet09} identified the critical mass for the problem with the Newtonian potential, $\K = c_d\abs{x}^{d-2}$,  and $A(u) = u^{m^\star}$.
The authors show that if $M < M_c$ then the solution exists globally and if $M > M_c$ then the solution may blow up in finite time. There $M_c$ is identified as 
\begin{equation*}
M_c = \left( \frac{2}{(m^\star-1)C_{m^\star} c_d} \right)^{1/(2-m^\star)},
\end{equation*}
where $C_{m^\star}$ is the best constant in the Hardy-Littlewood-Sobolev inequality given below in Lemma \ref{lem:GHLS}.
It is natural to ask the same question for more general cases. In this work we generalize these results to include inhomogeneous kernels and general nonlinear diffusion. 
First, we state the generalization of the finite time blow up results. \\

\begin{theorem}[Finite Time Blow Up for Critical Problems: $m^\star > 1$] \label{thm:SupercritMass}
Let $\D$ either be bounded and convex with a smooth boundary or $\D = \Real^d$. Let $\K$ and $A(u)$ be admissible and satisfy
\begin{itemize}
\item[\textbf{(B1)}] $\K(x) = c\abs{x}^{-d/p} + o(\abs{x}^{-d/p})$ as $x \rightarrow 0$ for some $c > 0$ and $d/(d-2) \leq p < \infty$. 
\item[\textbf{(B2)}] $x\cdot\grad \K(x) \leq -(d/p)\K(x) + C_1$ for all $x \in \Real^d$, for some $C_1 \geq 0$.
\item[\textbf{(B3)}] $A^\prime(z) = m\overline{A}z^{m-1} + o(z^{m-1})$ as $z \rightarrow \infty$ for some $m > 1,\overline{A} > 0$. 
\item[\textbf{(B4)}] $A(z) \leq (m-1)\Phi(z)$ for all $z > R$, for some $R > 0$ . 
\end{itemize}
Suppose the problem is critical, that is $m = m^\star$. 
Then the critical mass $M_c$ satisfies
\begin{equation*}
M_c = \left( \frac{2\overline{A}}{(m^\star-1)C_{m^\star}c} \right)^{1/(2-m^\star)},
\end{equation*}
and for all $M > M_c$ there exists a solution to \eqref{AE} which blows up in finite time with $\norm{u_0}_1 = M$.
\end{theorem}

\begin{theorem}[Finite Time Blow Up for Supercritical Problems] \label{thm:SupercritDiff}
Let $\D$ be as in Theorem \ref{thm:SupercritMass}. Let $\K$ satisfy \textbf{(B1)} and \textbf{(B2)} in Theorem \ref{thm:SupercritMass} and $A(u)$ satisfy \textbf{(B3)} and \textbf{(B4)} in Theorem \ref{thm:SupercritMass} with $1 < m < m^\star$. Then for all $M > 0$ there exists a solution which blows up in finite time with $\norm{u_0}_1 = M$. 
\end{theorem}

\noindent
The Newtonian and Bessel potentials both satisfy these conditions with $C_1=0$ (Lemma 2.2, \cite{SugiyamaDIE06}), and so the results apply to PKS with degenerate diffusion.  
Due to the decay of admissible kernels (Definition \ref{def:admK}) condition \textbf{(B2)} should only impose a significant restriction on the behavior of $\K$ at the origin. 
Power-law diffusion satisfies conditions \textbf{(B3)} and \textbf{(B4)}; however, \textbf{(B4)} is also restrictive, for example, $A(u) = u^m - u$ for $u$ large does not satisfy the condition.\\
\\
\noindent
The accompanying global existence theorem is significantly more inclusive than the blow up theorems, both in the kinds of kernels and nonlinear diffusion considered. As in Theorem \ref{thm:SupercritMass}, 
the estimate of the critical mass only depends on the leading order term of an asymptotic expansion of the kernel at the origin and the growth of the entropy at infinity. 
The approach used here and in \cite{Blanchet09,BlanchetEJDE06} relies on using the energy dissipation inequality \eqref{EnrDiss} and the continuation theorem (Theorem \ref{thm:Continuation}). The third key component is an inequality which relates the interaction energy $\mathcal{W}(u)$ to the entropy $S(u)$. For $m^\star > 1$ this is the Hardy-Littlewood-Sobolev inequality given in Lemma \ref{lem:GHLS}.  In this case, the estimate of the critical mass is given by \eqref{cond:critical_mass_powerlaw}. 

\begin{theorem}[Global Well-Posedness for $m^\star > 1$] \label{thm:GWP}
Suppose $m^\star > 1$. Then we have the following:
\begin{itemize}
\item[(i)] If the problem is subcritical, then the solution exists globally (i.e. $T_\star = \infty$) and is uniformly bounded in the sense $u \in L^\infty((0,\infty)\times D)$.
\item[(ii)] If the problem is critical then there exists a critical mass $M_c > 0$  such that if $\norm{u_0}_1 = M < M_c$, then the solution exists globally and is uniformly bounded in the sense $u \in L^\infty((0,\infty)\times D)$.  The critical mass is estimated below in \eqref{cond:critical_mass_powerlaw}. 
\end{itemize}
\end{theorem}

\begin{proposition}[Critical Mass For $m^\star > 1$]\label{prop:critical_mass}
If $\K = c\abs{x}^{-d/p} + o(\abs{x}^{-d/p})$ as $x \rightarrow 0$ for some $c \geq 0$ and $p$, $d/(d-2) \leq p < \infty$, then $M_c$ satisfies, 
\begin{equation}
\lim_{z\rightarrow\infty}\frac{\Phi(z)}{z^{m^\star}} - \frac{C_{m^\star}}{2}c M_c^{2-m^\star} = 0. \label{cond:critical_mass_powerlaw}
\end{equation} 
If $c = 0$ or $\lim_{z \rightarrow \infty}\Phi(z)z^{-m^{\star}} = \infty$ then we define $M_c = \infty$. 
\end{proposition}

\begin{remark}
By Lemma \ref{lem:LwpChar}, if $\K \in \wkspace{p}_{loc}$ then $\exists \, \delta, C > 0$ such that $\forall x, \, \abs{x} < \delta$, $\K(x) \leq C\abs{x}^{-d/p}$. 
Then, if the kernel does not admit an asymptotic expansion as in Proposition \ref{prop:critical_mass}, the critical mass $M_c$ can be estimated by,
\begin{equation*}
\lim_{z\rightarrow\infty}\frac{\Phi(z)}{z^{m^\star}} - \frac{C_{m^\star}}{2}C M_c^{2-m^\star} = 0.  
\end{equation*}
\end{remark}

\begin{remark}
Note, $\lim_{z \rightarrow\infty} \Phi(z)z^{-m^\star}$ is always well-defined but is not necessarily finite unless 
\begin{equation*}
\limsup_{z \rightarrow \infty}A^\prime(z)z^{1-m^\star} < \infty.
\end{equation*}
 If the problem is critical then necessarily $\lim_{z \rightarrow \infty} \Phi(z)z^{-m^\star} > 0$ so there always exists a positive mass which satisfies \eqref{cond:critical_mass_powerlaw}. Moreover, if the problem is subcritical then necessarily $\lim_{z \rightarrow \infty}\Phi(z)z^{-m^\star} = \infty$.
\end{remark}

\noindent
The case $m^\star = 1$ is analogous to the classical PKS problem in 2D, where linear diffusion is critical.
For the 2D PKS, the critical mass is given by $M_c = 8\pi$ for both the Newtonian and Bessel potentials \cite{BlanchetEJDE06,Calvez}.
In this work we treat the $m^\star = 1$ case for $d\geq 2$ on bounded domains, recovering the critical mass of the classical PKS, although 
\textbf{(D3)} technically requires the diffusion to be nonlinear and degenerate. 
The case $d \geq 3$ and $m^\star = 1$ is approached in \cite{KarchSuzuki10}, but the optimal critical mass is not identified. 
Our estimate is given below in \eqref{cond:critical_mass_2D}. As above, the critical mass only depends on the asymptotic expansion of the kernel at the origin and the growth of the entropy at infinity.
 We first state the analogue of Theorem \ref{thm:SupercritMass}. 

\begin{theorem}[Finite Time Blow Up for Critical Problems $m^\star = 1$] \label{thm:SupercritMass_2D}
Let $D$ be a smooth, bounded and convex domain and $d\geq 2$. Suppose 
 $\K$ satisfies
\begin{itemize}
\item[\textbf{(C1)}] $\K(x) = -c\ln\abs{x} + o(\ln\abs{x})$ as $x \rightarrow 0$ for some $c > 0$ . 
\item[\textbf{(C2)}] $x\cdot\grad \K(x) \leq -c + C\abs{x}$ for all $x \in \Real^d$, for some $C \geq 0$ .
\item[\textbf{(C3)}] $A(z) \leq \overline{A}z$ for some $\overline{A} > 0$. 
\end{itemize}
Then the critical mass $M_c$ satisfies
\begin{equation*}
M_c = \frac{2d\overline{A}}{c},
\end{equation*}
and for all $M > M_c$ there exists a solution which blows up in finite time with $\norm{u_0}_1 = M$. 
\end{theorem}

\noindent
The corresponding global existence theorem includes more general kernels and nonlinear diffusion. The proof is similar to Theorem \ref{thm:GWP}, except that the logarithmic Hardy-Littlewood-Sobolev inequality (Lemma \ref{lem:log_sob}) is used in place of the Hardy-Littlewood-Sobolev inequality.

\begin{theorem}[Global Well-Posedness for $m^\star = 1$ on Bounded Domains] \label{thm:GWP_2D}
Suppose $m^\star = 1$ and $d\geq 2$, let $D$ be bounded, smooth and convex. Then we have the following:
\begin{itemize} 
\item[(i)] If the problem is subcritical, then the solution exists globally and is uniformly bounded in the sense $u \in L^\infty((0,\infty)\times D)$.
\item[(ii)] If the problem is critical then there exists a critical mass, $M_c > 0$, such that if $\norm{u_0}_1 = M < M_c$, then the solution exists globally and is uniformly bounded in the sense $u \in L^\infty((0,\infty)\times D)$.
The critical mass is estimated below in \eqref{cond:critical_mass_2D}. 
\end{itemize}
\end{theorem}

\begin{proposition}[Critical Mass for $m^\star = 1$ on Bounded Domains] \label{prop:crit_mass_2D}
If $\K(x) = -c\ln\abs{x} + o(\ln\abs{x})$ as $x \rightarrow 0$ for some $c \geq 0$, then $M_c$ satisfies, 
\begin{equation}
\lim_{z\rightarrow\infty}\frac{\Phi(z)}{z \ln z} - \frac{c}{2d} M_c = 0. \label{cond:critical_mass_2D}
\end{equation}
If $c = 0$ or $\lim_{z \rightarrow \infty}\Phi(z)(z \ln z)^{-1} = \infty$ then we define $M_c = \infty$. 
\end{proposition}

\begin{remark}
By \textbf{(BD)} and \textbf{(MN)}, $\exists \, \delta, C > 0$ such that $\forall x, \, \abs{x} < \delta$, $\K(x) \leq -C\ln x$. 
Therefore, if the kernel does not have the asymptotic expansion required in Proposition \ref{prop:crit_mass_2D} then
 the critical mass $M_c$ may be estimated as, 
\begin{equation*}
\lim_{z\rightarrow\infty}\frac{\Phi(z)}{z \ln z} - \frac{C}{2d}M_c = 0. 
\end{equation*}
\end{remark}

\begin{remark} These theorems include many known global existence and finite time blow up results in the literature including \cite{SugiyamaADE07,SugiyamaDIE06,SugiyamaADE07,BertozziSlepcev10,Blanchet09,Kowalczyk05,CalvezCarrillo06}. Our main contributions to the existing theory is 
the unification of these results and the estimate of the critical mass for inhomogeneous kernels and general nonlinear diffusion. 
In the case of the Newtonian potential Blanchet et al. showed in \cite{Blanchet09} that solutions at the critical mass also exist globally. See \cite{Dolbeault04,Biler06,Blanchet08} for the corresponding result for classical 2D PKS. 
\end{remark}

\subsection{Properties of Admissible Kernels} \label{sec:PropAdmKern}
Definition \ref{def:admK} implies a number of useful characteristics which we state here and reserve the proofs for the Appendix \ref{Apx:Kernels}. First, we have that every admissible kernel is at least as integrable as the Newtonian potential. 
\begin{lemma} \label{lem:Newt}
Let $\K$ be admissible. Then $\grad \K \in \wkspace{d/(d-1)}$. If $d\geq 3$, then $\K \in \wkspace{d/(d-2)}$.
\end{lemma}

\noindent
In general, the second derivatives of admissible kernels are not locally integrable, but we may still properly define $\conv{D^2\K}{u}$ as a linear operator which involves a Cauchy principal value integral. By the Calder\'on-Zygmund inequality (see e.g. [Theorem 2.2 \cite{LittleStein}]) we can conclude that this distribution is bounded on $L^p$ for $1 < p < \infty$. The inequality also provides
an estimate of the operator norms, which is of crucial importance to the proof of uniqueness. 
 
\begin{lemma} \label{lem:CZ}
Let $\K$ be admissible and $\vec{v} = \conv{\grad \K}{u}$. Then $\forall p$, $1 < p < \infty$, $\exists \, C(p)$ such that $\norm{\grad \vec{v}}_p \leq C(p)\norm{u}_p$ and $C(p) \lesssim p$ for $2 \leq p < \infty$. 
\end{lemma}

\noindent
One can further connect the integrability of the kernel with the integrability of the derivatives at the origin, 
which provides a natural extension of Lemma \ref{lem:CZ} through the Young's inequality for $\wkspace{p}$. 
\begin{lemma} \label{lem:CZ_HighReg}
Let $d\geq 3$ and $\K$ be admissible. Suppose $\gamma$ is such that $1 < \gamma < d/2$. Then $\K \in \wkspace{d/(d/\gamma - 2)}_{loc}$ if and only if $D^2\K \in \wkspace{\gamma}_{loc}$. The same holds for $\grad \K \in \wkspace{d/(d/\gamma - 1)}_{loc}$. In particular, $m^\star = 1 + 1/\gamma - 2/d$ for some $1 < \gamma < d/2$ if and only if $D^2\K\in \wkspace{\gamma}_{loc}$.
Moreover, $m^\star = 1$ if and only if $D^2\K \in \wkspace{d/2}_{loc}$. 
\end{lemma}

\noindent
The following lemma clarifies the connection between the critical exponent and the interaction energy. 

\begin{lemma} \label{lem:GHLS}
Consider the Hardy-Littlewood-Sobolev type inequality, 
for all $f \in L^p$, $g \in L^q$ and $\K \in \wkspace{t}$ for $1 < p,q,t < \infty$ satisfying $1/p + 1/q + 1/t = 2$,
\begin{equation}
\abs{\int\int f(x)g(y)\K(x-y)dxdy} \lesssim \norm{f}_p\norm{g}_q \wknorm{\K}{t}. \label{ineq:GGHLS} 
\end{equation}
See \cite{Lieb83}. In particular, if $(p+1)/p = m^\star > 1$, then for all $u \in L^1 \cap L^{m^\star}$,
\begin{align}
\int u(x)u(y)\abs{x-y}^{-d/p}dxdy \leq C_{m^\star}\norm{u}_1^{2-m^\star}\norm{u}_{m^\star}^{m^\star}. \label{ineq:GHLS}
\end{align}
Here $C_{m^\star}$, depending only on $p$ and $d$, is taken to be the best constant for which \eqref{ineq:GHLS} holds for all such $u$. 
\end{lemma}

\begin{remark}
It is not necessarily the case that $C_{m^\star}$ is easily related to the optimal constant in \eqref{ineq:GGHLS}. 
It is shown in \cite{Blanchet09} that $C_{2-2/d}$ is achieved for a fairly explicit family of extremals, but to our knowledge, extremals of \eqref{ineq:GHLS} have not been constructed for other values of $m^\star$. 
\end{remark}

\noindent
If $m^\star = 1$ then we will need the logarithmic Hardy-Littlewood-Sobolev inequality, as in for instance \cite{Dolbeault04,Blanchet08}.
\begin{lemma}[Logarithmic Hardy-Littlewood-Sobolev inequality \cite{CarlenLoss92}] \label{lem:log_sob}
Let $d \geq 2$ and $0 \leq f \in L^1$ be such that $\abs{\int f \ln f dx} < \infty$. Then, 
\begin{equation}
-\int\int_{\Real^d\times \Real^d}f(x)f(y) \ln\abs{x-y} dxdy \leq \frac{\norm{f}_1}{d}\int_{\Real^d} f\ln f dx + C(\norm{f}_1). \label{ineq:log_sob}
\end{equation}
\end{lemma}

\section{Uniqueness} \label{sec:Unique}
\noindent
We now prove the uniqueness of weak solutions stated in Theorem \ref{thm:Uniqueness}. 
\begin{proof}({\bf Theorem \ref{thm:Uniqueness}})
The proof follows \cite{BertozziBrandman10,BertozziSlepcev10} and estimates the difference of weak solutions in $\dot{H}^{-1}$, motivated by the fact that the nonlinear diffusion is monotone in this norm \cite{VazquezPME}. To this end, if the domain is bounded, we define $\phi(t)$ as the zero mean strong solution of
\begin{align}
\Delta \phi(t) = u(t) - v(t) \textup{ in } \D \label{def:phi} \\
\grad \phi(t) \cdot \nu = 0, \textup{ on } \partial \D, \label{cond:phi_noflux}
\end{align}
where $\nu$ is the outward unit normal of $\D$. If the domain is $\Real^d$ for $d \geq 3$, we let $\phi(t) = -\conv{\mathcal{N}}{(u - v)}$ where $\mathcal{N}$ is the Newtonian potential.
In either case, by the integrability and boundedness of weak solutions $u(t)$ and $v(t)$ we can conclude $\phi(t) \in L^\infty(D_T) \cap C([0,T];\dot{H}^1)$, $\grad \phi(t) \in L^\infty(D_T)\cap L^2(D_T)$ and $\phi_t$ solves, 
\begin{equation*}
\Delta \phi_t = \partial_t u - \partial_t v. 
\end{equation*}
Then since $\norm{u(t) - v(t)}_{\dot{H}^{-1}} = \norm{\grad \phi(t)}_2$, we will show that $\norm{\grad \phi(t)}_2 = 0$. 
During the course of the proof, we integrate by parts on a variety of quantities. If the domain is bounded, then the boundary terms 
will vanish due to the no-flux conditions \eqref{cond:no_flux},\eqref{cond:phi_noflux}. In $\Real^d$, the computations are justified as $\grad \K \ast u, \grad A(u), \grad \K \ast v, \grad A(v), \grad \phi \in L^2(D_T)$. \\

\noindent
By the regularity of $\phi(t)$ and the no-flux boundary conditions (\ref{cond:phi_noflux}), \eqref{cond:no_flux} we have possibly up to a set of measure zero,
\begin{align*}
\frac{1}{2}\frac{d}{dt} \int \abs{\grad \phi(t)}^2 dx & = <\grad \phi(t), \partial_t \grad \phi(t) >
= -<\partial_tu(t) - \partial_tv(t),\phi(t)>. 
\end{align*}
Therefore, using $\phi(t)$ in the definition of weak solution and \eqref{cond:phi_noflux} we have,
\begin{align*}
\frac{1}{2} \frac{d}{dt} \int \abs{\grad \phi(t) }^2 dx & = \int (\grad A(u(t)) - \grad A(v(t))) \cdot \grad\phi(t) dx \\
& -\int (u-v)(\conv{\grad \K}{u})\cdot \grad \phi dx
 - \int v (\conv{\grad \K}{(u-v)}) \cdot \grad \phi dx. \\
& := I_1 + I_2 + I_3. 
\end{align*}
We drop the time dependence for notational simplicity. 
Since $A$ is increasing, we have the desired monotonicity of the diffusion,
\begin{equation*}
I_1 = - \int \left(A(u) - A(v)\right)(u-v) dx \leq 0.
\end{equation*}
We now concentrate on bounding the advection terms. \\

\noindent
We follow \cite{BertozziSlepcev10}. By integration by parts we have,
\begin{align}
I_2 & = \sum_{i,j}\int \partial_i\phi (\partial_{ij}\K \ast u) \partial_j \phi dx
 + \sum_{i,j}\int \partial_i \phi (\partial_j \K \ast u) \partial_{ij}\phi dx. \label{eq:I2}
\end{align}
If the domain is bounded, we may apply integration by parts, 
\begin{align*}
\sum_{i,j}\int \partial_i \phi (\partial_j\K \ast u) \partial_{ij}\phi dx & = -\sum_{i,j}\int \partial_{ij}\phi \partial_j\K \ast u \partial_i \phi dx - \sum_{i,j}\int \partial_i\phi (\partial_{jj}\K \ast u)\partial_i \phi dx \\
& + \sum_{i,j}\int_{\partial \D} \abs{\partial_i \phi}^2 \partial_j\K \ast u \nu_j dS,
\end{align*}
where $\nu$ is the unit outward normal to $\D$. As in \cite{BertozziSlepcev10}, we have $\nabla\K\ast u \cdot \nu \leq 0$ on $\partial \D$ since $\D$ is convex and $\K$ is radially decreasing, so that term is non-positive. If the domain were $\Real^d$, such boundary terms would vanish. Therefore by integration by parts again we have,
\begin{equation*}
\sum_{i,j}\int \partial_i \phi (\partial_j\K \ast u) \partial_{ij}\phi dx \leq -\frac{1}{2}\int (\Delta \K \ast u)\abs{\grad \phi}^2 dx, 
\end{equation*}
which together with \eqref{eq:I2} implies,
\begin{equation*}
I_2 \lesssim \int \abs{D^2\K \ast u} \abs{\grad \phi}^2 dx.
\end{equation*}
By H\"older's inequality, Lemma \ref{lem:CZ} and 
$\grad \phi \in L^\infty(D_T)$ for $p \geq 2$,
\begin{align}
\int \abs{\conv{D^2\K}{u}} \abs{\grad \phi}^2 dx & \leq \norm{\conv{D^2\K}{u}}_p \left(\int \abs{\grad \phi}^{2p/(p-1)}dx \right)^{(p-1)/p} \nonumber \\
 &\lesssim p\norm{u}_p\norm{\grad \phi}_\infty^{2/p}\left(\int \abs{\grad \phi}^{2}dx\right)^{(p-1)/p} \nonumber \\
 & \lesssim p\left( \int \abs{\grad \phi}^{2}dx \right)^{(p-1)/p}, \label{ineq:I2}
\end{align}
where the implicit constant depends only on the uniformly controlled $L^p$ norms of $u$ and $v$. \\

\noindent
As for $I_3$, we compute as in \cite{BertozziSlepcev10}. By the computations in the proof of Lemma \ref{lem:CZ} we may justify
integration by parts on the inside of the convolution, that is, 
\begin{equation*}
\norm{\sum_{j}\int \partial_{i}\K(x-y) \partial_{jj} \phi dx}_{2} \lesssim \norm{\grad \phi}_{2}. 
\end{equation*}
which by Cauchy-Schwarz implies, 
\begin{align} 
 I_3 & \lesssim \norm{v}_{\infty}\norm{\grad \phi}_2^2. \label{ineq:I3}
\end{align}

\noindent
Letting $\eta(t) = \int \abs{\grad \phi(t)}^2 dx$, \eqref{ineq:I2} and \eqref{ineq:I3} imply the differential inequality,
\begin{equation*}
\frac{d}{dt}\eta(t) \leq \hat{C}p\max(\eta(t)^{1-1/p},\eta(t)),
\end{equation*}
where $\hat{C}$ again depends only on the uniformly controlled $L^p$ norms of $u,v$.
The differential equality does not have a unique solution, but all of the solutions are absolutely continuous integral solutions bounded above by the maximal solution $\overline{\eta}(t)$. 
By continuity, for $t < 1/\hat{C}$ the maximal solution is given by $\overline{\eta}(t) = (\hat{C}t)^p$, hence, 
\begin{equation*}
\eta(t) \leq \overline{\eta}(t) = (\hat{C} t)^p.  
\end{equation*}
For $t < 1/(2\hat{C})$ we then have 
\begin{equation*}
\eta(t) \leq \overline{\eta}(t) \leq 2^{-p},  
\end{equation*}
and we take $p \rightarrow \infty$ to deduce that for $t \in [0,1/(2\hat{C}))$, $\eta(t) = 0$, therefore the solution is unique. This procedure may be iterated to prove uniqueness over the entire interval of existence since the time interval only depends on uniformly controlled norms.
\end{proof}

\section{Local Existence} \label{sec:LocExist}
\subsection{Local Existence in Bounded Domains}

Let $\tilde{A}(z)$ be a smooth function on $\Real^+$ such that $\tilde{A}'(z)>\eta$ for some $\eta>0$.  In addition, let $\vec{v}$ be a given smooth velocity field with bounded divergence.  Classical theory gives a global smooth solution to the uniformly parabolic equation 
\begin{align}\label{CAE}
u_t = \Delta \tilde{A}(u) - \nabla \cdot \left(u\vec{v}\right)
\end{align}
\noindent   (see \cite{Lieberman96}).  The solutions obey the global $L^\infty$ bound  
\begin{align}\label{RegLinfty}
\left\|u\right\|_{L^\infty(D)}\leq \left\|u_0\right\|_{L^\infty(D)}e^{\left\|\left(\nabla\cdot \vec{v}\right)_{-}\right\|_{L^\infty(\pd)}t}.
\end{align}  
  We take advantage of this theory to prove existence of weak solutions to \eqref{AE} by regularizing the degenerate diffusion and the kernel.  Consider the modified aggregation equation
\begin{align}\label{RAE}
\reg{u}_t =  \Delta \reg{A}(\reg{u})- \nabla \cdot \left(\reg{u}\left(\nabla \mollify{\kernel} \ast \reg{u}\right)\right),    
\end{align}
with corresponding no-flux boundary conditions \eqref{cond:no_flux}.  We define 
\begin{align}\label{ParabolicReg}
\reg{A}(z) = \int^z_0 a_\epsilon'(z) dz,
\end{align}
\noindent where $a_\epsilon'(z)$ is a smooth function, such that $A'(z)+\epsilon \leq a_\epsilon'(z) \leq A'(z)+2\epsilon$, and the standard mollifier is denoted $\mathcal{J_\epsilon}v$.  We first prove existence of solutions to the regularized equation \eqref{RAE}, this is stated formally in the following proposition.
\begin{proposition}[Local Existence for the Regularized Aggregation Diffusion Equation]\label{LERADD}
Let $\epsilon>0$ be fixed and $u_0(x)\in C^{\infty}(\overline{D})$, then \eqref{RAE} has a classical solution $u$ on $\pd$ for all $T>0$.
\end{proposition}
\noindent We obtain the proof of Proposition \ref{LERADD} directly from Theorem 12 in \cite{BertozziSlepcev10}.  The proof requires a bound on $\norm{\nabla A^\epsilon}_{L^2(\pd)}$, for some $T>0$.  We state this lemma for completeness but reference the reader to \cite{BertozziSlepcev10} for a proof.  

\begin{lemma}[Uniform Bound on Gradient of $A(u)$]\label{GradAbdd}
Let $\epsilon>0$ be fixed and $u^{\epsilon}\in L^\infty(\pd)$ be a solution to \eqref{RAE}.  There exists a constant $C = C(T,\left\|\nabla \mollify{\kernel}\ast u^{\epsilon}\right\|_{L^\infty(D)}, \left\|u^{\epsilon}\right\|_\infty)$ such that:
\begin{align}\label{bbgradA}
\left\|\nabla A^\epsilon(u^{\epsilon})\right\|_{L^2(\pd)}\leq C.
\end{align}
\end{lemma}
\begin{remark}
The estimate given by \eqref{bbgradA} is independent of $\epsilon$.  
\end{remark}

\noindent Proposition \ref{LERADD} gives a family of solutions $\{u^{\epsilon}\}_{\epsilon>0}$.  To prove local existence to the original problem \eqref{AE} we first need some a priori estimates which are independent of $\epsilon$.  Mainly, we obtain an independent-in-$\epsilon$ bound on the $L^\infty$ norm of the solution and the velocity field.  This is the main difference in the local existence theory from
\cite{BertozziSlepcev10}.  Due to the singularity of the kernels significantly more is required to obtain these a priori bounds.  We first state a lemma, due to Kowalczyk \cite{Kowalczyk05} and extended to $d > 2$ and $\Real^d$ in \cite{CalvezCarrillo06}. The proof is based on the Alikakos iteration. 

\begin{lemma}[Iteration Lemma \cite{Kowalczyk05,CalvezCarrillo06}] \label{dynamic}
Let $0 < T\leq \infty$ and assume that there exists a $c>0$ and $u_c>0$ such that $A'(u)>c$ for all $u>u_c$. Then if $\left\|\nabla \kernel\ast u\right\|_{\infty} \leq C_1$ on $[0,T]$ then $\left\|u\right\|_\infty \leq C_2(C_1)\max\{1,M,\left\|u_0\right\|_\infty\}$ on the same time interval. 
\end{lemma}

\begin{lemma}[$L^\infty$ Bound of Solution] \label{LinftyG}
Let $\set{u^\epsilon}_{\epsilon > 0}$ be the classical solutions to \eqref{RAE} on $\pd$, with smooth, non-negative, and bounded initial data $\mollify{u_0}$. 
Then there exists $C = C(\norm{u_0}_1,\norm{u_0}_\infty)$ and $T = T(\norm{u_0}_1,\norm{u_0}_p)$ for any $p > d$ such that for all $\epsilon > 0$, 
\begin{align}\label{Linfty}
\left\|u^\epsilon(t)\right\|_{L^\infty(D)}\leq C
\end{align}
\noindent for all $t\in[0,T]$.
\end{lemma}
\begin{proof}
For simplicity we drop the $\epsilon$.  The first step is to obtain an interval for which the $L^p$ norm of $u$ is bounded.  Following the work of \cite{JagerLuckhaus92} we define the function $\reg{u}_{k}=(\reg{u}-k)_+$, for $k>0$.  Due to conservation of mass the following inequality provides a bound for the $L^p$ norm of $u$ given a bound on the $L^p$ norm of $u_k$,  
\begin{align}\label{UkvsU}
\norm{u}_p^p \leq C(p)(k^{p-1}\norm{u}_1 + \norm{u_k}_p^p).  
\end{align}  
We look at the time evolution of $\left\|u_k\right\|_{p}$ and make use of the parabolic regularization \eqref{ParabolicReg}. \\
\\
{\it Step 1:}  
\begin{align*}
\frac{d}{dt}\left\|u_k\right\|_p^p &= p\int u_k^{p-1}\nabla \cdot\left(\nabla \reg{A}(u) - u\nabla \mollify{\kernel}\ast u\right) dx \\
& = -p(p-1)\int A^{\epsilon'} \nabla u_k \cdot\nabla u dx - p(p-1)\int uu^{p-2}_k\nabla \mollify{\kernel}\ast u\; dx.  \\ 
&\leq  -\frac{4(p-1)}{p} \int A'(u)\left|\nabla u_k^{p/2}\right|^2dx + p(p-1)\int u_k^{p-1}\nabla u_k \cdot \nabla \mollify{\kernel}\ast u\;dx\\
&\quad + kp(p-1)\int u_k^{p-2}\nabla u_k\cdot \nabla \mollify{\kernel}\ast u\;dx,
\end{align*}
where we used the fact that for $l>0$
\begin{align}\label{UeqUk}
u(u_k)^l = (u_k)^{l+1} + ku_k^l.  
\end{align}
\noindent Hence, integrating by parts once more gives
\begin{align*}
\frac{d}{dt}\left\|u_k\right\|^p_p & \leq \frac{4(p-1)}{p} \int A'(u)\left|\nabla u_k^{p/2}\right|^2dx - (p-1) \int u_k^p \Delta \mollify{\K}\ast u dx- kp\int u_k^{p-1} \Delta \mollify{\K}\ast u dx\\
& \leq -C(p)\hspace{-5pt}\int \hspace{-5pt} A'(u)\left|\nabla u_k^{p/2}\right|^2dx + C(p)\left\|u_k\right\|_{p+1}^{p}\left\|\Delta \mollify{\K} \ast u\right\|_{p+1} + C(p) k\left\|u_k\right\|_{p}^{p-1}\left\|\Delta \mollify{\K} \ast u\right\|_{p} \\
& \leq -C(p) \hspace{-5pt} \int \hspace{-5pt}A'(u)\left|\nabla u_k^{p/2}\right|^2dx + C(p)\left(\left\|u_k\right\|_{p+1}^{p+1} + \left\|u\right\|^{p+1}_{p+1}\right) + C(p)k\left(\left\|u_k\right\|_{p}^{p} + \left\|u\right\|^p_{p}\right).
\end{align*}
In the last inequality we use Lemma \ref{lem:CZ}.  Now, using \eqref{UkvsU} we obtain that
\begin{align*}
\frac{d}{dt}\left\|u_k\right\|_p^p dx &\leq -C(p)\hspace{-5pt}\int \hspace{-5pt} A'(u)\left|\nabla u_k^{p/2}\right|^2dx + C(p)\left\|u_k\right\|_{p+1}^{p+1} + C(p,k)\left\|u_k\right\|_{p}^{p} + C(p,k,M).
\end{align*} 
An application of the Gagliardo-Nirenberg-Sobolev inequality gives that for any $p$ such that $d<2(p+1)$ (see Lemma \ref{lem:GNS} in the Appendix):
\begin{align*}
\left\|u\right\|_{p+1}^{p+1}\lesssim \left\|u\right\|^{\alpha_2}_{p}\left\|u^{p/2}\right\|^{\alpha_1}_{W^{1,2}},
\end{align*}
where $\alpha_1 = d/p,\; \alpha_2 = 2(p+1)-d$.  From the inequality $a^rb^{(1-r)} \leq ra + (1-r)b$ (using that
$a= \delta \left\|u^{p/2}\right\|^2_{W^{1,2}}$ and $r=\alpha_1/2$) we obtain  
\begin{align*}
\left\|u\right\|_{p+1}^{p+1}&\lesssim \frac{1}{\delta^{\beta_1}}\left\|u\right\|^{\beta_2}_{p} + r\delta^2\left\|\nabla u^{p/2}\right\|^2_2 + r\delta^2 \norm{u}^p_p.
\end{align*}
Above $\beta_1,\;\beta_2 > 1$.  For $k$ large enough we have that $A'(u)>c>0$ over $\set{u > k}$; hence, if we choose $\delta$ small enough we obtain the final differential inequality:
\begin{align}\label{DILp}
\frac{d}{dt}\left\|u\right\|^p_p \lesssim C(p)\left\|u_k\right\|^{\beta_2}_p + C(p,k,r\delta) \left\|u_k\right\|_{p}^{p} + C(p,k,\left\|u_0\right\|_{1}).
\end{align}
The inequality \eqref{DILp} in turns gives a $T_p=T(p)>0$ such that $\left\|u_k\right\|_p$ is bounded on $[0,T_p]$.  Inequality \eqref{UkvsU} gives that $\left\|u\right\|_p$ remains bounded on the same time interval.  Next we prove that the velocity field is bounded in $L^\infty(D)$ on some time interval $[0,T]$.  This then allows us invoke Lemma \ref{dynamic} and obtain the desired bound. \\
\\
{\it Step 2:}\\
\\
\noindent 
Since $\grad \K \in L_{loc}^1$ and $\grad \K \mathbf{1}_{\Real^d \setminus B_1(0)} \in L^q$ for all $q > d/(d-1)$ (by Lemma \ref{lem:Newt}), we have for all $p > d/(d-1)$, 
\begin{equation*}
\norm{\vec{v}}_{p} = \norm{\conv{\grad \K}{u}}_p \leq \norm{\grad \K \mathbf{1}_{B_1(0)}}_1\norm{u}_p + \norm{\grad \K \mathbf{1}_{\Real^d\setminus B_1(0)}}_pM.
\end{equation*}
By Lemma \ref{lem:CZ} we also have, for all $p, 1 < p < \infty$, 
\begin{equation*}
\norm{\grad \vec{v}}_{p} = \norm{\conv{D^2\K}{u}}_p \lesssim \norm{u}_p. 
\end{equation*} 
By Morrey's inequality we have $\vec{v} \in L^\infty(D_T)$ by choosing some $p > d$ and invoking step one, and Lemma \ref{dynamic} concludes the proof. Note that the bound depends on the geometry of the domain through the constant on the Gagliardo-Nirenberg-Sobolev inequality (Lemma \ref{lem:GNS}). However, this constant is related to the regularity of the domain, and not directly to the diameter of the domain.
\end{proof}
\noindent In addition to the a priori estimates the proof of Theorem \ref{LEADD} requires precompactness of $\{u^\epsilon\}_{\epsilon > 0}$ in $L^1(\pd)$.

\begin{lemma}[Precompactness in $L^1(\Omega_T)$]\label{precomp}   
The sequence of solutions obtained via Proposition \ref{LERADD}, $\{u_\epsilon\}_{\epsilon > 0}$, which exist on $[0,T]$, is precompact in $L^1(\pd)$.  
\end{lemma}
\noindent The proof of Lemma \ref{precomp} follows exactly the work in \cite{BertozziSlepcev10}.  The key is to prove that the sequence satisfies the Riesz-Frechet-Kolmogorov Criterion.  This relies on the fact that $\norm{A(u^\epsilon)}_{L^2(0,T;H^1(D))} \leq C$ uniformly.     

\begin{proof}(\textbf{Theorem \ref{LEADD}})
For a  given $\epsilon>0$, if we regularize the initial condition $\reg{u}_0(x) = \mollify{u_0(x)}$, Proposition \ref{LERADD} gives a solution $\reg{u}$ to \eqref{RAE}.  Furthermore, the proof of Proposition \ref{LERADD} and Lemma \ref{LinftyG} gave uniform-in-$\epsilon$ bounds on $\left\|\reg{A}(u)\right\|_{L^2(0,T,H^1(D))}$, $\left\|\reg{u}\right\|_{L^\infty(\pd)}$, and $\left\|\reg{u}_t\right\|_{L^2(0,T,H^{-1}(D))}$.  By Lemma
\ref{LinftyG}, all solutions exist on $[0,T]$, with $T$ independent of $\epsilon$.  Also, recalling that $\reg{A}(z)\geq A(z)$ and $a'_\epsilon(z)\geq A'(z)$ gives that
\begin{align*}
\left\|A(\reg{u})\right\|_{L^2(0,T,H^1(D))}\leq C,
\end{align*}
\noindent where $C$ is independent of $\epsilon$.  Since $L^2(0,T,H^1(D))$ is weakly compact there exists a $\rho$ such that some subsequence of $\left\{\reg{u} \right\}_{\epsilon>0}$ converges weakly, i.e $A(u^{\epsilon_j})\rightharpoonup \rho$ in $L^2(0,T,H^1(D))$. Precompactness in $L^1$ implies strong convergence of $u^{\epsilon_j}$ to some $u\in L^1(\pd)$; therefore,  $A(u) =\rho$.  In fact, the $L^\infty(\pd)$ bound on $u^{\epsilon_j}$ gives strong convergence in $L^p(\pd)$, for $1\leq p<\infty$, via interpolation.  Also, Young's inequality gives  

\begin{align}\label{L1ConvoConve}
\left\|u^{\epsilon_j}\nabla \mollifyj{\kernel}\ast u^{\epsilon_j}- u\nabla \kernel\ast u\right\|_{L^1(\pd)}& \leq \left\|u\right\|_{L^\infty(\pd)}\left\|\nabla \mollifyj{\kernel} \ast u^{\epsilon_j}-\nabla\kernel \ast u\right\|_{L^1(\pd)}\notag \\
&\quad+ \left\|\nabla \mollifyj{\kernel} \ast u^{\epsilon_j}\right\|_{L^\infty(\pd)}\left\|u^{\epsilon_j}-u\right\|_{L^1(\pd)}\notag \\
& \lesssim \left(\left\|u\right\|_{L^\infty(\pd)}\left\|\nabla \kernel\right\|_{L^1_{loc}} + \left\|\nabla \mollifyj{\kernel} \ast u^{\epsilon_j}\right\|_{L^\infty(\pd)}\right)\left\|u^{\epsilon_j}-u\right\|_{L^1(\pd)}\notag \\
&\quad +\norm{u}_\infty\norm{u^{\epsilon_j}}_{\infty}\norm{\grad \mollifyj \K - \grad \K}_{L^1_{loc}}. 
\end{align}

\noindent Therefore, by interpolation $u$ satisfies \eqref{WF2}.  Furthermore, we obtain that $u\in C([0,T];H^{-1}(D))$.  To prove that $u(t)$ is continuous with respect to the weak $L^2$ topology one uses standard density arguments.  Since $D$ is a bounded, $u$ is therefore also continuous in the weak $L^1$ topology.  To prove continuity in the strong $L^2$ topology we define $F(z) = \int_0^z A(s)ds$ and show that it is continuous in the strong $L^1$ topology.  Indeed,    Lemma \ref{L6} in the Appendix, see \cite{BertozziSlepcev10} for a proof, gives 
\begin{align} \label{ContofF}
\lim_{h\rightarrow 0} \left|\int (F(u(t))-F(u(t+h)))\; dx\right| = \lim_{h\rightarrow 0} \int^{t+h}_t <u_\tau,A(\tau)> \;d\tau.
\end{align}   
\noindent Recall that $\norm{A(u)}_{L^\infty(\pd)}\leq A(\norm{u}_{L^\infty(\pd)})$  and so $A(u) \in L^2(0,T,H^{-1}(D))$.  Therefore, the left hand side of \eqref{ContofF} goes to 0 as $h \rightarrow 0$.  Now, we can invoke Lemma \ref{L9} in Appendix, \cite{BertozziSlepcev10}, to obtain that $u\in C([0,T]; L^2(D))$.  Using interpolation the $L^\infty$ bound of $u$ gives that $u\in C([0,T];L^p(D))$, for $1\leq p<\infty$. 
\end{proof}

\subsection{Local Existence in $\Real^d$}
Now we consider solutions to \eqref{AE} in $\Real^d$ for $d\geq 3$.  We obtain such solution by taking the limit of the solutions in balls centered on the origin with increasing radius $n$, denoted by  $\ball$. 

\begin{proof} ({\bf Theorem \ref{LEADDRd}})
Let $\ball$ be defined as above and consider the truncation of the initial condition on $\ball$, i.e. $u^n_0 = \mathbf{1}_{\ball}u_0$.  By Theorem \ref{LEADD}, we have a family of solutions $\{u_n\}_{n>0}$ on $\ball$ for all $t\in [0,T]$.  Define a new sequence, $\{\tilde{u}_n\}_{n>0}$, where $\tilde{u}_n$ is the zero extension of $u_n$.  The previous work for bounded domains gives the uniform bounds
\vspace{-6pt}
\begin{align}\label{AprioriRd}
\left\|\tilde{u}_n\right\|_{L^\infty(\pdRd)}\quad\;&\leq C_1,\\
\left\|\nabla A(\tilde{u}_n)\right\|_{L^2(\pdRd)}&\leq C_2.  
\end{align}
The bounds may be taken independent of $n$ since the constant in the Gagliardo-Nirenberg-Sobolev inequality, Lemma \ref{lem:GNS}, does not depend directly on the diameter of the domain and may be taken uniform in $n \rightarrow \infty$. \\

\noindent
Therefore, there exist $u,w\in L^2(\pdRd)$ for which $\tilde{u}_n\rightharpoonup u$ and $\nabla A(\tilde{u}_n)\rightharpoonup w$ in $L^2(\pdRd)$.  Furthermore, \eqref{AprioriRd}  implies $\left\|u\right\|_{L^\infty(\pdRd)} \leq C_1.$   
Precompactness of $\left\{\tilde{u}^\epsilon_n \right\}_{\epsilon>0}$ in $L^1(\ball)$ for fixed $n>0$ and Theorem 2.33 in \cite{Adams03} gives that $\{\tilde{u}_n\}_{n>0}$ is precompact in $L^1_{loc}(\pdRd)$.  Therefore, up to a subsequence, not renamed, $\tilde{u}_n \rightarrow u$ in $L^1_{loc}(\pdRd)$; thus, $w=\nabla A(u)$. Also, the $L^\infty$ bound gives that $\tilde{u}_n\rightarrow u$ in $L^p_{loc}(\Real^d)$ for $1\leq p<\infty$.  \\
\\
In addition, we have the estimate
\begin{align}\label{L2convBdd}
\left\|\tilde{u}_n\nabla \kernel\ast \tilde{u}_n\right\|_{L^2(\Real^d_T)}&\leq \left\|\nabla \kernel \ast\tilde{u}_n\right\|_{L^\infty(\Real^d_T)} \left\|\tilde{u}_n\right\|_{L^2(\Real^d_T)}.\end{align}
Therefore, we can extract a subsequence that converges weakly to some $w_1 \in L^2(\Real^d_T)$.  Since $u\mathbf{1}_{\ball} \in L^\infty(0,T,L^1(\Real^d))$ and $u\mathbf{1}_{\ball} \nearrow u$ by monotone convergence $u\in L^\infty(0,T,L^1(\Real^d_T))$.  Once again, from the estimates performed in the bounded domains $\tilde{u}_n\nabla\kernel\ast \tilde{u}_n\rightarrow u\nabla\kernel\ast u$ in $L^1_{loc}(\pdRd)$.  Therefore, we can identify $w_1 = u\nabla \kernel \ast u$.  \\
\\
We now show that $u\in C([0,T];L^1_{loc}(\Real^d))$, which we know to be true, implies that $u \in C([0,T]; L^1(D))$.  Let $t_n \rightarrow t \in [0,T]$ then for all $R > 0$ we have,
\begin{equation} \label{ContL1}
\int \abs{u(t_n) - u(t)}dx = \int_{B_R} \abs{u(t_n) - u(t)}dx +
\int_{\Real^d\setminus B_R} \abs{u(t_n) - u(t)} dx.
\end{equation}
The first term on the right hand side of \eqref{ContL1} can be bounded by $\epsilon/2$, provided $n$ is chosen large enough, since $u\in C([0,T];L^1_{loc}(\Real^d))$.  To bound the second term we first show that $A(u) \in L^1(\pdRd)$.  By \textbf{(D3)} we can deduce $\lim_{z \rightarrow 0}A(z)z^{-1} = 0$. 
Then, for $k>0$ there exists some $0 < C_k < \infty$ such that if $z < k$ then $A(z) \leq Cz$.  Hence, 
\begin{align*}
\int A(u) dx & = \int_{\set{ u < k}}A(u) dx + \int_{\set{u \geq k}} A(u)dx
\\
& \leq CM + A(\norm{u}_\infty) \lambda_u(k) < \infty.
\end{align*}
\noindent Therefore, $\norm{A(u)}_{L^1(\pdRd)} \leq C(M,\norm{u}_\infty)T$.  Now, 
let $w(x)$ be a smooth radially-symmetric cut-off function with $w(x) = 0$
for $\abs{x} < 1/2$ and $w(x) = 1$ for $\abs{x} \geq 1$. Then consider the
quantity, $M_R(t) = \int u w(x/R) dx$. Then formally,
\begin{equation*}
\frac{d}{dt}M_R(t) = \frac{1}{R}\int u v\cdot (\grad w)(x/R)dx +
\frac{1}{R^2} \int A(u) (\Delta w)(x/R) dx.
\end{equation*}
\noindent Estimating terms in $L^\infty$ gives,
\begin{equation*}
\frac{d}{dt}M_R(t) \lesssim \frac{\norm{v}_\infty \norm{u}_1}{R} +
\frac{1}{R^2} \int A(u) dx.
\end{equation*}
\noindent Formally, then
\begin{equation}\label{NoMassEscape}
M_R(t) \lesssim M_R(0) + M\norm{v}_{L^1((0,t);L^\infty)}R^{-1} +
\norm{A(u)}_{L^1( (0,t) \times \Real^d)} R^{-2}.
\end{equation}
\noindent Since $A \in L^1((0,t)\times \Real^d)$ and $M_R(0)
\rightarrow 0$ as $R\rightarrow \infty$, by choosing $R$ sufficiently large, the last term of \eqref{ContL1} can be bounded by $\epsilon/2$.  Hence, implies that $u \in C([0,T]; L^1(\Real^d))$.  Furthermore, via interpolation we obtain that $u \in C([0,T]; L^p(\Real^d))$ for $1 \leq p<\infty$. \\
\\
Conservation of mass can be proved similarly using a cut-off function $w(x)=1$ for $\abs{x}\leq 1/2$ and $w(x)=0$ for $\abs{x}\geq 1$, see the proof of Theorem 15 in \cite{BertozziSlepcev10} for a similar proof. 
\end{proof}
\noindent We are left to prove the energy dissipation inequality \eqref{EnrDiss}.  As expected, the approach is to regularize the energy and take the limit in the regularizing parameters. 

\begin{proof} ({\bf Proposition \ref{prop:EnrDiss}})
Define
\begin{align*}
h(u) = \int_1^u \frac{A'(s)}{s}ds,
\end{align*}
then $\Phi (u) = \int_0^u h(s) ds$. The regularized entropy is defined similarly with $a'_\epsilon(u)$, as defined in \eqref{ParabolicReg}, taking the place of $A'(u)$.  Given a smooth solution $\reg{u}$ to \eqref{RAE} one can verify,
\begin{align}\label{RegDissEq}
\energy_\epsilon (\reg u(t)) + \int_0^t \int \frac{1}{\reg{u}}\abs{a'_\epsilon(\reg{u} )\nabla \reg{u} - \reg{u}\nabla\mollify{\K}\ast \reg{u}}^2 dx d\tau = \energy_\epsilon (\reg u(0)).
\end{align}
Here $\energy_\epsilon(u(t))$ denotes the free energy with the regularized entropy and kernel. 
Once again we take the limit $\epsilon$ approaches zero to obtain \eqref{EnrDiss}.  We first show that the entropy converges.  \\
\\
 {\it Step 1}:  The parabolic regularization gives
\begin{align*}
h(z) + \epsilon \ln z \leq h_\epsilon (z) \leq h(z) + 2\epsilon \ln z & \quad\text{for} \; 1 \leq z,\\
h(z) + 2\epsilon \ln z \leq h_\epsilon (z) \leq h'(z) + \epsilon \ln z & \quad\text{for} \; z\leq 1.
\end{align*}
\noindent Therefore, writing $\Phi(u) = \int_0^1 h(s) ds + \int_1^uh(s) ds$ one observes that 
\begin{align}\label{BoundsOnPhi}
\Phi(u)-2\epsilon \leq \Phi_\epsilon(u) \leq \Phi(u)+ 2\epsilon (u\ln u)_+.
\end{align}
This will allow us to show convergence of the entropy.  In fact,  
\begin{align*}
\abs{\int \Phi_\epsilon(\reg{u}) -\Phi(u)dx} &\leq \int \abs{ \Phi_\epsilon(\reg{u}) - \Phi(\reg{u}) }dx + \int \abs{ \Phi(\reg{u}) - \Phi(u) }dx\\
\tiny{\eqref{BoundsOnPhi}}& \leq 2\epsilon \int (1+ \reg{u}\ln \reg{u})_+ dx + \norm{\Phi}_{C^1([0,\norm{\reg{u}}_\infty])}\int\abs{\reg{u}-u}dx.\\
& \leq 2\epsilon \left(\abs{D}+\norm{\ln \reg{u}}_\infty\norm{\reg{u_0}}_1\right) + C\norm{\reg{u}-u}_1.    
\end{align*}
\noindent Conservation of mass, boundedness of smooth solutions, and precompactness in $L_{loc}^1$ imply there exists a subsequence, such that as $\epsilon_j \rightarrow 0$, 
\begin{align*}
\int \Phi_{\epsilon_j}(\reg{u}_j)dx \rightarrow \int \Phi(u)dx.  
\end{align*}
\noindent {\it Step 2:}  To show convergence of the interaction energy we need that for $a.e\;t\in (0,T)$ 
\begin{align}\label{L1convEnergy}
\int \reg{u}(t)\mollify{\kernel}\ast \reg{u}(t)dx \rightarrow \int u(t) \kernel\ast u(t) dx.
\end{align}  
 \noindent  Since $\kernel \in L^1_{loc}(D)$ we know that $\left\|\kernel \ast u\right\|_{L^\infty}$ is bounded; hence, replacing $\nabla \kernel$ with $\kernel$ in \eqref{L1ConvoConve} gives the desired result.   Finally, we are left to deal with the entropy production functional.   \\
\\
\noindent {\it Step 3}:  From Lemma 10 in \cite{CarrilloEntDiss01}, 
\begin{align}\label{EPTConv}
\int \frac{1}{u}\abs{A'(u)\nabla u - u\nabla\K\ast u}^2 dx \leq \liminf_{\epsilon\rightarrow 0} \int \frac{1}{\reg{u}}\abs{a'_\epsilon(\reg{u} )\nabla \reg{u} - \reg{u}\nabla\mollify{\K}\ast \reg{u}}^2 dx.
\end{align}
\noindent  We also note that this was proved in \cite{BertozziSlepcev10}.  The proof of \eqref{EPTConv} relies on a result due to Otto in \cite{Otto01}, refer to Lemma \ref{WLS} in the Appendix.  In our case, $\reg{u}\in L^1(\pd)$ and $J_\epsilon =\nabla \reg{A}(\reg{u})-\reg{u}\nabla \kernel\ast \reg{u} \in L^1_{loc}(\pd).$  Furthermore, up to a sequence not renamed, $\reg{u}\rightharpoonup u \in L^2$ and $J_\epsilon \rightharpoonup J$ in $L^2$, therefore, we can apply Lemma \ref{WLS}.  \\
\\
For the energy dissipation estimate in $\Real^d$ we again consider the family of solutions
$\left\{u_r\right\}$ to \eqref{AE}  on $B_r$ (for simplicity let $u_r$ denote the zero-extension of the solutions).  Since $u_n(0) \mathbf{1}_{B_n} \nearrow u(0)$ by monotone convergence we obtain that $\energy(u_n(0))\rightarrow \energy(u(0))$.  Noting that $\kernel \in L^{d/(d-2)}$ allows us to make a modification to \eqref{L2convBdd} and obtain that $u_n\kernel\ast u_n\rightharpoonup u\kernel\ast u$ in
$L^2(\Real^d_T)$.  Furthermore, \eqref{L1convEnergy} implies that $u_n\kernel\ast u_n\rightarrow u \kernel\ast u$ in $L^1_{loc}$.  We are left to verify the uniform integrability over all space.  First note that Morrey's inequality implies
\begin{align*}
\left\|\kernel \ast \tilde{u}_n\right\|_\infty& \lesssim \left\|\nabla \kernel \ast u\right\|_\infty + \left\|\kernel \ast u_n\right\|_p\\
& \leq \left\|\nabla \kernel \ast u\right\|_\infty + \wknorm{\kernel}{d/(d-2)}  \left\|u_n\right\|_{dp/(d+2p)}.
\end{align*}
\noindent Hence, taking $p$ sufficiently large we obtain that $\kernel\ast u_n$ is bounded in $L^\infty(D_T)$.  Therefore, 
\begin{align*}
\int_{\Real^d\setminus B_k}\hspace{-12pt} u_n\kernel\ast u_n dx&\leq \left\|\kernel\ast u_n\right\|_\infty \int_{\Real^d\setminus B_k} u_n dx.
\end{align*}
\noindent 
This fact along with \eqref{NoMassEscape} gives that for any $\epsilon>0$ there exists a $k_\epsilon$ sufficiently large such that for all $k>k_\epsilon$ 
\begin{align*}
\int_{\Real^d\setminus B_k}\hspace{-12pt} \tilde{u}_n\kernel\ast \tilde{u}_ndx\leq \epsilon.
\end{align*}
\noindent This gives convergence of the interaction energy.  The result follows from the weak lower semi-continuity of the entropy production functional and $\int \Phi(u) dx$ in $L^2$.  
\end{proof}
\section{Continuation Theorem} \label{sec:ContThm}

Continuation of weak solutions, Theorem \ref{thm:Continuation}, is a straightforward consequence of the local existence theory and the following lemma, which follows substantially the recent work in \cite{Blanchet09,Kowalczyk05,BlanchetEJDE06}. This lemma provides a more precise version of Lemma \ref{LinftyG} and has a similar proof. 
\begin{lemma} \label{lem:UnifInt}
Let $\set{u^\epsilon}_{\epsilon > 0}$ be the classical solutions to \eqref{RAE} on $D_T$, with non-negative initial data $\mollify{u_0}$. 
Suppose there exists $T_0$, $0 < T_0 \leq \infty$, such that
\begin{equation}
\sup_{\epsilon > 0} \lim_{k \rightarrow \infty}\sup_{t \in (0,T_0)}\norm{(u^\epsilon-k)_+}_{\frac{2-m}{2-m^\star}} = 0, \label{cond:unif_int}
\end{equation}
where $m$ is such that $1 \leq m \leq m^\star$ and $\lim \inf_{z \rightarrow \infty}A^{\prime}(z)z^{1-m} > 0$.  Then there exists $C = C(M,\norm{u_0}_{\infty})$ such that for all $\epsilon > 0$,
\begin{equation*}
\sup_{t \in (0,T_0)}\norm{u^\epsilon(t)}_\infty \leq C. 
\end{equation*}
In particular, if $T_0 = \infty$, then $\set{u^\epsilon}_{\epsilon > 0}$ are uniformly bounded for all time, and therefore the weak solution $u(t)$, is uniformly bounded for all time.   
\end{lemma}

\begin{proof}(\textbf{Lemma \ref{lem:UnifInt}})
\noindent
Let $\overline{q} = (2-m)/(2-m^\star) \geq 1$. It will be convenient to define $\gamma$, $1 \leq \gamma \leq d/2$ such that $m^\star = 1 + 1/\gamma - 2/d$. 
We first bound intermediate $L^p$ norms over the same interval, $(0,T_0)$.  Then we use Morrey's inequality and Lemma \ref{dynamic} to finish the proof. \\
\\
\textit{Step 1:}\\
\\
\noindent We have two cases to consider, $m^{\star} = 2 - 2/d$ and $m^\star < 2- 2/d$, which occurs if $D^2K \in \wkspace{\gamma}_{loc}$ for $\gamma > 1$ (Lemma \ref{lem:CZ_HighReg}).  In the former we show that for any $p \in (\overline{q}, \infty)$ we have $u^\epsilon(t)$ uniformly bounded in $L^{\infty}\left(0,T_0;L^p \right)$. 
In the latter case we only show that for $\overline{q} < p \leq \gamma/(\gamma-1)$ we 
have $u^\epsilon(t)$ uniformly bounded in $L^{\infty}\left( 0,T_0;L^p \right)$. In either case, this is sufficient to apply Lemma \ref{dynamic} and conclude the proof. \\    

\noindent
Let $k > 0$ be some constant to be determined later and let $u_k = (u - k)_+$.  We have dropped
the $\epsilon$ and time dependence for notational convenience.
By conservation of mass and \eqref{UkvsU}, it suffices to control $\norm{u_k}_p$ for any $k > 0$. Thus, using the parabolic regularization, \eqref{ParabolicReg}, and \eqref{UkvsU} we obtain
\begin{equation*}
\frac{d}{dt}\norm{u_k}_p^p \leq -p(p-1)\int u_k^{p-2} A^\prime(u)\abs{\grad u}^2dx + p(p-1)\int (u_k^{p-1} + k u_k^{p-2})\grad u \cdot \mollify \grad \conv{\K}{u} dx.
\end{equation*}
Then,
\begin{align}
\frac{d}{dt}\norm{u_k}_p^p & \leq -4(p-1)\int A^\prime(u)\abs{\grad u_k^{p/2}}^2dx - \int ((p-1) u_k^{p} + kp u_k^{p-1}) \mollify \Delta \conv{\K}{u} dx. \label{ineq:MainEstimate} 
\end{align}
 
\noindent Since the constants are not relevant, we treat the cases together only noting minor differences when they appear. 
If $m = 2-2/d$ we may use H\"older's inequality and then Lemma \ref{lem:CZ} to obtain a bound on the first term from the advection:
\begin{equation*}
\abs{\int u_k^p \mollify \conv{\Delta \K}{u} dx} \lesssim_{p,\K} \norm{u_k}_{p+1}^p\norm{u}_{p+1}. 
\end{equation*}
On the other hand, if $\gamma > 1$ we have from the generalized Hardy-Littlewood-Sobolev inequality \eqref{ineq:GGHLS} (Lemma \eqref{lem:GHLS}),
\begin{equation*}
\abs{\int u_k^p \mollify \conv{\Delta \K}{u} dx} \lesssim_{p,\K} \norm{u_k}_{\alpha p}^p\norm{u}_t + C(M)\norm{u_k}^p_p,
\end{equation*}
with the scaling condition $1/\alpha + 1/t + 1/\gamma = 2$. 
Choosing $t = \alpha p$ implies that
\begin{equation}
\frac{1}{\alpha} = \frac{2 - 1/\gamma}{1 + 1/p}. \label{eq:alpha}
\end{equation}
Notice that from our choice of $p$ then $1 \leq 1/p + 1/\gamma$; thus, $1/\alpha \leq 1$.  Note that in the case when $m = 2 - 2/d$ then $t = \alpha p = p+1$. 
Thus we estimate the advection terms, 
\begin{align}
\abs{\int u_k^p \mollify \conv{\Delta \K}{u} dx} & \lesssim_{p,\K} \norm{u_k}_{\alpha p}^{p}\norm{u}_{\alpha p} + C(M)\norm{u_k}_p^p \nonumber \\
& \lesssim \norm{u_k}_{\alpha p}^{p+1} + \norm{u}_{\alpha p}^{p+1} + C(M)\norm{u_k}_p^p \nonumber \\
\ineqtext{\eqref{UkvsU}} & \lesssim \norm{u_k}_{\alpha p}^{p+1} + C(M)\norm{u_k}_p^p + C(k,M). \label{ineq:adv_1} 
\end{align}
The lower order terms in the advection can be controlled using H\"older's inequality and Lemma \ref{lem:CZ}, 
\begin{align}
\abs{\int u_k^{p-1} \mollify \conv{\Delta \K}{u} dx} & \lesssim_p \norm{u_k}_p^{p-1}\norm{u}_p \nonumber \\
 &\leq \norm{u_k}_{p}^{p} + \norm{u}_p^p \nonumber \\
\ineqtext{\eqref{UkvsU}} & \lesssim \norm{u_k}_p^p + C(k,M). \label{ineq:adv_2}
\end{align}
We now aim to compare the dissipation term in \eqref{ineq:MainEstimate} with the estimates \eqref{ineq:adv_1} and \eqref{ineq:adv_2}. 
We use the Gagliardo-Nirenberg-Sobolev inequality (Lemma \ref{lem:GNS}),
\begin{equation}
\norm{u_k}_{\alpha p} \lesssim \norm{u_k}^{\alpha_2}_{\overline{q}}\norm{u_k^{(p+m-1)/2}}_{W^{1,2}}^{\alpha_1} 
\end{equation}
with 
\begin{equation*}
\alpha_1 = \frac{2d}{p}\left(\frac{(p - \overline{q}/\alpha)}{ \overline{q}(2 - d) + dp + d(m-1)}\right),
\end{equation*}
and
\begin{equation*}
\alpha_2 = 1 - \alpha_1(p+m -1)/2 > 0.
\end{equation*}
By the definition of $\overline{q}$ and (\ref{eq:alpha}) we have that, 
\begin{equation}
\alpha_1(p+1)/2 = 1, \label{ineq:alpha1}
\end{equation}
which implies,
\begin{equation}
\norm{u_k}^{p+1}_{\alpha p} \lesssim \norm{u_k}^{\alpha_2(p+1)}_{\overline{q}}\left(\int u_k^{m-1}\abs{\grad u_k^{p/2}}^2 dx + \int u_k^{p+m-1} dx \right). \label{ineq:GNSCont2}
\end{equation}
If $d = 2$ then necessarily $m = m^\star = 1$ and this inequality will be sufficient. However, for $d \geq 3$, more work must be done. Define, 
\begin{equation*}
I = \int u_k^{m-1} \abs{\grad u_k^{p/2}}^2 dx. 
\end{equation*}
Then, for $\beta_1 \leq \alpha_1$ and $(p+m-1)\beta_1/2 < 1$, 
\begin{equation*}
\beta_1 = \frac{2d(1 - \overline{q}/(p+m-1))}{ \overline{q}(2 - d) + dp + d(m-1)},
\end{equation*}
and $\beta_2 = 1 - \beta_1(p+m-1)/2 > 0$, we have the following by Lemma \ref{lem:GNS}, 
\begin{align*}
\int u_k^{p+m-1} dx &\lesssim \norm{u_k}_{\overline{q}}^{(p+m-1)\beta_2}\left(I + \int u_k^{p+m-1} dx \right)^{(p+m-1)\beta_1/2} \\
& \lesssim \norm{u_k}_{\overline{q}}^{(p+m-1)\beta_2}\left(I^{(p+m-1)\beta_1/2}  + \left(\int u_k^{p+m-1} dx\right)^{(p+m-1)\beta_1/2}\right). 
\end{align*}
Therefore, by weighted Young's inequality for products,
\begin{align}
\int u_k^{p+m-1} dx \lesssim \norm{u_k}_{\overline{q}}^{(p+m-1)\beta_2}\left(1 + I \right) + \norm{u_k}_{\overline{q}}^{\gamma_0}, \label{ineq:GNS_inhomg}
\end{align}
for some $\gamma_0 > 0$, the exact value of which is not relevant. Putting \eqref{ineq:GNSCont2} and \eqref{ineq:GNS_inhomg} together implies, 
\begin{equation}
\norm{u_k}_{\alpha p}^{p+1} \lesssim \mathcal{P}(\norm{u_k}_{\overline{q}})I + C(\norm{u_k}_{\overline{q}}), \label{ineq:GNS_full}
\end{equation}
where $\mathcal{P}(z)$ denotes a polynomial such that $\mathcal{P}(z) \rightarrow 0$ as $z \rightarrow 0$.
By definition of $m$, $\exists\; \delta > 0$ such that for $k$ sufficiently large then $u > k$ implies $A^\prime(u) > \delta u^{m-1}$.Therefore, combining \eqref{ineq:MainEstimate} with \eqref{ineq:GNS_full},\eqref{ineq:adv_1} and \eqref{ineq:adv_2} implies, 
\begin{align*}
\frac{d}{dt}\norm{u_k}_p^p & \leq -C(p)\delta\int u_k^{m-1} \abs{\grad u_k^{p/2}}^2 dx + C(p)\norm{u_k}_{\alpha p}^{p+1}  \nonumber \\
& + C(M,p)\norm{u_k}_p^p + C(k,M,p) \nonumber \\
& \leq -\frac{C(p)\delta}{\mathcal{P}(\norm{u_k}_{\overline{q}})}\norm{u_k}_{\alpha p}^{p+1} + C(p)\norm{u_k}_{\alpha p}^{p+1} \nonumber \\
&  + C(M,p)\norm{u_k}_p^p + C(k,M,p,\norm{u_k}_{\overline{q}}).  
\end{align*}
By interpolation against $L^1$, conservation of mass and $\alpha \geq 1$ we have 
\begin{equation*}
\norm{u_k}_p^p \lesssim_M 1 + \norm{u_k}_{p \alpha}^{p+1}. 
\end{equation*}
Therefore, by assumption \eqref{cond:unif_int} we may choose $k$ sufficiently large such that there exists some $\eta > 0$ which satisfies the following for all $t \in (0,T_0)$, 
\begin{equation*}
\frac{d}{dt}\norm{u_k}_p^p \leq -\eta\norm{u_k}_p^p + C(k,M,p,\norm{u_k}_{\overline{q}}). 
\end{equation*} 
It follows that $\norm{u_k}_p$ is bounded uniformly on $(0,T_0)$. \\
\\
\textit{Step 2:} \\
\\
\noindent
The control of these $L^p$ norms will enable us to invoke Lemma \ref{dynamic} and conclude $u^\epsilon(t)$ is bounded uniformly in $L^\infty(D_{T_0})$. 
Since $\grad \K \in L_{loc}^1$ and $\grad \K \mathbf{1}_{\Real^d \setminus B_1(0)} \in L^q$ for all $q > d/(d-1)$ (by Lemma \ref{lem:Newt}), we have for any $q > d/(d-1)$
\begin{equation*}
\norm{\vec{v}}_{q} = \norm{\conv{\grad \K}{u}}_{q} \leq \norm{\grad \K \mathbf{1}_{B_1(0)}}_1\norm{u}_q + \norm{\grad \K \mathbf{1}_{\Real^d\setminus B_1(0)}}_qM.
\end{equation*}
If $\gamma > 1$, then we may choose $q \in (d/(d-1),\gamma/(\gamma - 1)]$, since in this case necessarily $d \geq 3$. Otherwise we may choose $q > d/(d-1)$ arbitrarily. 
Then, step one implies $\vec{v} \in L^\infty((0,T_0);L^q)$. 
If $\gamma > 1$ then, noting that Definition \ref{def:admK} implies $D^2\K \mathbf{1}_{\Real^d\setminus B_1(0)} \in L^q$ for all $q > 1$, 
\begin{equation*}
\norm{\grad \vec{v}}_{d+1} = \norm{\conv{D^2\K}{u}}_{d+1} \leq \wknorm{D^2\K \mathbf{1}_{B_1(0)}}{\gamma} \norm{u}_{p} + \norm{\grad \K \mathbf{1}_{\Real^d\setminus B_1(0)}}_{d+1}M,
\end{equation*}
for $p = \gamma(d+1)/(d(\gamma - 1) + 2\gamma - 1)$. Note that
\begin{equation*}
1 < p = \frac{\gamma(d+1)}{d(\gamma - 1) + 2\gamma - 1} \leq \frac{\gamma}{\gamma - 1}.
\end{equation*}
On the other hand, if $m^\star = 2 - 2/d$ then the above proof shows that $u^\epsilon(t)$ is bounded uniformly in $L^\infty((0,T_0);L^p)$ for all $p < \infty$. Therefore, by Lemma \ref{lem:CZ} we have $\norm{\grad \vec{v}}_p \lesssim \norm{u}_p \lesssim 1$, for all $1 < p < \infty$. 
In either case, this is sufficient to apply  Morrey's inequality and conclude that $\norm{\vec{v}}_\infty$ is uniformly bounded on $(0,T_0)$. By Lemma \ref{dynamic} we then have that $u^\epsilon$ is uniformly bounded in $L^\infty(D_{T_0})$ and we have proved the lemma. As in Lemma \ref{LinftyG}, the uniform bounds depend on the domain but not it's diameter. 
\end{proof}

\begin{remark} \label{rmk:subcritical}
The proof of this lemma directly implies global well-posedness
in the subcritical case since \eqref{cond:unif_int} is only necessary in the critical and supercritical cases. 
 Moreover, in the critical case, one may prove directly that there exists some $M_0$ such that if $M < M_0$ the solution is global.  However, $M_0$ will
generally depend on the constant of the Gagliardo-Nirenberg-Sobolev inequality, as in \cite{SugiyamaADE07,SugiyamaDIE07,JagerLuckhaus92}. As discussed in the recent works of \cite{Blanchet09,BlanchetEJDE06}, the use of a continuation theorem will allow for a more accurate estimate of the critical mass through the use
of the free energy. 
\end{remark}

\noindent

\begin{proof}({\bf Theorem \ref{thm:Continuation}})
Suppose, for contradiction, that the weak solution cannot be continued past $T_\star < \infty$ and \eqref{cond:equint} fails. 
As the regularized problems are bounded, this implies the hypotheses of Lemma \ref{lem:UnifInt} are satisfied on $(0,T_\star)$, and therefore $\sup_{\epsilon > 0} \sup_{t \in (0,T_\star)}\norm{u^\epsilon(t)}_{p} \leq \eta$ as $t \nearrow T_\star$ for some $p > \overline{q}$ and $\eta > 0$. 
By the proof of Lemma \ref{LinftyG}, for any $\eta > 0$ there exists a $\tau = \tau(\eta,M) > 0$ such that 
if $\norm{u_0}_{p} < \eta$ then $\norm{u^\epsilon}_{p} \leq C$ for all $\epsilon > 0$. 
Therefore, we may choose some $t_n < T_\star$ such that $\tau$ satisfies $t_n + \tau > T_\star$ and, by Theorems \ref{LEADD} and \ref{LEADDRd}, we construct a solution 
$\tilde{u}(x,t)$ on the time interval $[t_n,t_n + \tau)$. By uniqueness, $\tilde{u}(x,t) = u(x,t)$ a.e. for $t \in [t_n,T_\star)$; hence, it is a genuine extension of the original solution $u(x,t)$.  However, it exists on a longer time interval which is a contradiction. 
\end{proof}

\section{Global Existence} \label{sec:GE}
We now prove Theorem \ref{thm:GWP}. 
We first note that the entropy is bounded below uniformly in time, which is a consequence of assumption \textbf{(D3)} of Definition \ref{def:admDiff}.
\begin{lemma} \label{lem:entropy_lowerbd}
Let $u(x,t)$ be a weak solution to \eqref{AE}. Then,
\begin{equation*}
\int \Phi(u(t)) dx \geq -CM. 
\end{equation*}
\end{lemma}
\begin{proof}
Let $h(z) = \int_1^z A^\prime(s)s^{-1} ds$. By Definition \ref{def:admDiff}, \textbf{(D3)}, for $z \leq 1$,
\begin{equation*}
h(z) \geq -C > -\infty. 
\end{equation*}
Therefore,
\begin{align*}
\int \Phi(u) dx =\int \int_0^u h(z) dz dx & \geq \int \mathbf{1}_{\set{u \leq 1}} \int_0^u h(z) dz + \mathbf{1}_{\set{u \geq 1}} \int_0^1 h(z) dz dx. \\
& \geq -\int \mathbf{1}_{\set{u \leq 1}}Cu - \mathbf{1}_{\set{u \geq 1}}C dx \\ 
& \geq -2C\norm{u}_1. 
\end{align*}
where the last line followed from Chebyshev's inequality. 
\end{proof}

\subsection{Theorem \ref{thm:GWP}: $m^\star > 1$}
\begin{proof}({\bf Theorem \ref{thm:GWP}})
We only prove the second assertion
under the hypotheses of Proposition \ref{prop:critical_mass}, as the subcritical case follows similarly. 
By the energy dissipation inequality \eqref{EnrDiss} we have for all time $0 \leq t < T_\star$, 
\begin{equation}
S(u(t)) - \mathcal{W}(u(t)) \leq \F(u_0) := F_0. \label{eq:endiss_gwp}
\end{equation}
We drop the time dependence of $u(t)$ for notational simplicity.
By the assumption on $\K$, $\forall \, \epsilon > 0$, $\exists \, \delta>0$ such that $\abs{\K(x)} \leq (c+\epsilon)\abs{x}^{-d/p}$ for $\abs{x} < \delta$. 
By Lemma \ref{lem:GHLS} we have,
\begin{equation*}
\int \Phi(u) dx - \frac{1}{2}C_{m^\star} M^{2-m^\star}(c+\epsilon)\norm{u}_{m^\star}^{m^\star} \leq F_0 + \frac{1}{2}\norm{\K|_{B_\delta(0)}}_\infty M^2, 
\end{equation*}
By \eqref{cond:critical_mass_powerlaw} and $M < M_c$, there exists $\epsilon > 0$ small enough and $\alpha,k > 0$ such that 
\begin{equation}
\Phi(z)z^{-m^\star} - \frac{1}{2}C_{m^\star} M^{2-m^\star}\left(c + \epsilon \right) \geq \alpha > 0, \textup{ for all } z > k. \label{ineq:asymp_gwp}
\end{equation}
By Lemma \ref{lem:entropy_lowerbd} we have,
\begin{equation*}
\int_{\set{u > k}}\hspace{-.5cm} u^{m^\star}\left(\Phi(u)u^{m^\star} - \frac{1}{2}C_{m^\star} M^{2-m^\star}\left(c + \epsilon \right) \right) dx -\frac{1}{2}\int_{\set{u < k}}\hspace{-.5cm}C_{m^\star} M^{2-m^\star}\left(c + \epsilon \right)u^{m^\star}dx  \leq F_0 + C(\delta,M), 
\end{equation*}
and by \eqref{ineq:asymp_gwp},
\begin{equation*}
\alpha \int_{\set{u > k}} \hspace{-.5cm}u^{m^\star}dx - \frac{1}{2}C_{m^\star} M^{2-m^\star}\left(c + \epsilon\right)\int_{\set{u < k}}\hspace{-.5cm}u^{m^\star}dx  \leq F_0 + C(M,\delta).
\end{equation*}
By mass conservation we have that $\norm{u}_{m^\star}$ is a priori bounded independent of time and Theorem \ref{thm:Continuation} and Lemma \ref{lem:UnifInt} implies global existence and uniform boundedness.
\end{proof}

\subsection{Theorem \ref{thm:GWP_2D}: $m^\star = 1$}
\noindent
The proof of Theorem \ref{thm:GWP_2D} follows similarly, but requires the logarithmic Hardy-Littlewood-Sobolev inequality (Lemma \ref{lem:log_sob}) as opposed to Lemma \ref{lem:GHLS}.  

\begin{proof}({\bf Theorem \ref{thm:GWP_2D}})
\noindent

\noindent
We only prove the second assertion
under the hypotheses of Proposition \ref{prop:crit_mass_2D}, as the subcritical case follows similarly. 
We will again use Theorem \ref{thm:Continuation} and prove
\begin{equation*}
\sup_{t \in (0,\infty)}\int (u \ln u)_+ dx < \infty.
\end{equation*}  

\noindent
By the energy dissipation inequality \eqref{EnrDiss} we again have \eqref{eq:endiss_gwp}.
By the assumptions of Proposition \ref{prop:crit_mass_2D}, 
for all $\epsilon > 0$ there exists $\delta > 0$ such that, 
\begin{equation*}
\int \Phi(u) dx + (c+\epsilon)\frac{1}{2}\int\int_{\abs{x - y} < \delta}\hspace{-.5cm} u(x)u(y)\ln\abs{x-y} dx dy \leq C(F_0,\delta,M).  
\end{equation*}
By $D$ bounded, the logarithmic Hardy-Littlewood-Sobolev inequality \eqref{ineq:log_sob} implies,  
\begin{equation*}
\int \Phi(u) dx - (c+\epsilon)\frac{M}{2d}\int u \ln u dx \leq C(F_0,\delta,M,\textup{diam} D).
\end{equation*}
Choosing $k > 0$ large and recalling Lemma \ref{lem:entropy_lowerbd} implies 
\begin{equation*}
\int_{\set{u > k}}\hspace{-.5cm} u \ln u \left(\frac{\Phi(u)}{u \ln u} - (c+\epsilon)\frac{M}{2d}\right) dx - (c+\epsilon)\int_{\set{u < k}}u \ln u dx \leq C(F_0,\delta,M,\textup{diam} D).
\end{equation*}
As in the proof of Theorem \ref{thm:GWP}, by conservation of mass, \eqref{cond:critical_mass_2D} and $M < M_c$, we may choose $\epsilon > 0$ small enough and $k$ large enough such that 
\begin{equation*}
\int_{\set{u > k}}\hspace{-.5cm} u \ln u dx \leq C(F_0,M,\textup{diam} D).  
\end{equation*}
\end{proof}

\section{Finite Time Blow Up} \label{sec:FTBU}
In this section we prove Theorem \ref{thm:SupercritDiff} and Theorem \ref{thm:SupercritMass}. We prove 
Theorem \ref{thm:SupercritDiff} as it is somewhat easier, though the technique is the same as that used to prove 
Theorem \ref{thm:SupercritMass}. 

\subsection{Supercritical Case: Theorem \ref{thm:SupercritDiff}}
For Theorem \ref{thm:SupercritDiff} we state the following lemma, which provides insight into the nature of the supercritical cases. 
The proof and motivation follows \cite{Blanchet09}. 

\begin{lemma} \label{lem:infF_Supercritical}
Define $\mathcal{Y_M} = \set{u \in L^1 \cap L^{m^\star}: u\geq 0, \norm{u}_1 = M}$. 
Suppose $\K$ satisfies \textbf{(B1)} and $A(u)$ satisfies \textbf{(B3)} for some $m > 1, \overline{A} > 0$. 
Suppose further that the problem is supercritical, that is, $m < m^\star$. Then $\inf_{\mathcal{Y_M}} \F = -\infty$. 
Moreover, there exists an infimizing sequence with vanishing second moments which converges to the Dirac delta mass in the sense of measures.  
\end{lemma}

\begin{proof}
Let $0 < \theta < 1$, $\alpha = d/p$. Then by Lemma \ref{lem:GHLS} there exists $h^\star$ such that,
\begin{equation}
\theta C_{m^\star} \leq \frac{\abs{\int\int h^\star(x)h^\star(y)\abs{x-y}^{-\alpha}dx dy}}{ \norm{h^\star}^{2-m^\star}_1\norm{h^\star}^{m^\star}_{m^\star} } \leq C_{m^\star}. \label{ineq:hstar_supercrit} 
\end{equation}
We may assume without loss of generality that $h^\star \geq 0$, since replacing $h^\star$ by $\abs{h^\star}$ will only increase the value of the convolution. By density, we may take $h^\star \in C_c^\infty$ and therefore with a finite second moment. \\

\noindent
Let $\mu = \norm{h^\star}_1^{1/d}M^{-1/d}$, $\lambda > 0$ and $h_\lambda(x) = \lambda^dh^\star(\lambda \mu x)$. First note, by \textbf{(B3)}, $\forall \epsilon > 0$, $\exists \, R > 0$ such that,
\begin{align}
\int\Phi(h_\lambda)dx & = \int \int_0^{h_\lambda}\int_1^s \frac{A^\prime(z)}{z} dz ds dx \nonumber \\
& \leq \int \int_0^{h_\lambda} \int_R^{\max(s,R)} (m\overline{A} + \epsilon)z^{m-2} dz + \int_1^R\frac{A^\prime(z)}{z} dz ds dx \nonumber \\
& \leq \frac{\overline{A} + \epsilon}{m-1}\norm{h_\lambda}_{m}^m + C(R)\norm{h_\lambda}_1. \label{ineq:hlambda_phi} 
\end{align}
By \textbf{(B1)} and $h^\star \in C_c^\infty$, $\forall \, \epsilon > 0, \, \exists \lambda > 0$ sufficiently large such that, 
\begin{equation}
-\mathcal{W}(t) \leq -(c-\epsilon)\frac{\mu^{-2d + \alpha}\lambda^{\alpha}}{2}\int\int h^\star(x)h^\star(y) \abs{x - y}^{-\alpha} dxdy. \label{ineq:hlambda_W}
\end{equation}
Combining \eqref{ineq:hlambda_W},\eqref{ineq:hlambda_phi} with \eqref{ineq:hstar_supercrit} and Lemma \ref{lem:GHLS}, we have for $\lambda,R$ sufficiently large, 

\begin{align*}
\F(h_\lambda) & \leq \frac{\lambda^{dm - d}M}{(m-1)\norm{h^\star}_1}(\overline{A} + \epsilon)\norm{h^\star}_m^m - \lambda^{\alpha}(\theta-\epsilon)\frac{C_{m^\star}}{2}\left(\frac{\norm{h^\star}_1}{M}\right)^{-2 + \alpha/d}\norm{h^\star}_1^{2-m^\star}\norm{h^\star}_{m^\star}^{m^\star} \\ 
& + C(R)\mu^{-d}\norm{h^\star}_1. 
\end{align*}
By supercriticality, we have $\alpha = dm^\star - d > dm - d$, 
and so for $\epsilon < \theta$, we take $\lambda \rightarrow \infty$ to conclude that for all values of the mass $M > 0$ we have $\inf_{\mathcal{Y_M}}\F = -\infty$. Moreover, since $h^\star \in C_c^\infty$, the second moment of $h_\lambda$ goes to zero and $h_\lambda$ converges to the Dirac delta mass in the sense of measures.
\end{proof}
\begin{proof}({\bf Theorem \ref{thm:SupercritDiff}})
We may justify the formal computations for weak solutions using the regularized problems and taking the limit but we do not include such details.  We treat both bounded and unbounded domains together pointing out the differences when they appear. 
Let 
\begin{equation*}
I(t) = \int \abs{x}^2 u(x,t) dx. 
\end{equation*}
If the domain is bounded then by \eqref{cond:no_flux},  
\begin{align}
\frac{d}{dt}I(t) & = 2d\int A(u) dx + 2 \int \int u(x)u(y) x \cdot \grad \K(x-y) dx dy - \int_{\partial \D} A(u) x\cdot \nu dS \nonumber \\
& = 2d\int A(u) dx + \int \int (x-y)\cdot\grad \K(x-y) u(x)u(y) dx dy - \int_{\partial \D} A(u) x\cdot \nu(x) dS, \label{SM}
\end{align}
where the second integral was obtained by symmetrizing in $x$ and $y$, the time dependence was dropped for notational simplicity and $\nu(x)$ denotes the outward unit normal of $\D$ at $x \in \partial D$.  By translation invariance and convexity of $D$, we may assume without loss of generality that $x \cdot \nu(x) \geq 0$. For the rest of the proof we may treat bounded domains and $\D = \Real^d$ together, since for each,
\begin{equation*}
\frac{d}{dt}I(t) \leq 2d\int A(u) dx + 2 \int \int u(x)u(y) x \cdot \grad \K(x-y) dx dy. 
\end{equation*}

\noindent
We use \textbf{(B2)} on $\K$, to obtain 
\begin{equation*}
\frac{d}{dt}I(t) \leq 2d\int A(u) dx - 2d/p\mathcal{W}(u) + C_1M^2.  
\end{equation*}
By \textbf{(D3)}, \textbf{(B4)} and Lemma \ref{lem:entropy_lowerbd},
\begin{align*}
\int A(u) dx & = \int_{\set{u < R}} A(u) dx + \int_{\set{u > R}} A(u) dx \\
& \leq C(M) + (m-1)\int_{\set{u > R}} \Phi(u) dx \\
& \leq C(M) + (m-1)\int \Phi(u) dx. 
\end{align*} 
Using that $2d(m-1) < 2d(m^\star-1) = 2d/p$ we have,  
\begin{equation*}
\frac{d}{dt}I(t) \leq 2d(m - 1)\F(u) + C(M,C_1).   
\end{equation*}
We use the energy dissipation inequality \eqref{EnrDiss} to bound the first term, 
\begin{equation*}
\frac{d}{dt}I(t) \leq 2d(m - 1)\F(u_0) +  C(M,C_1). 
\end{equation*}
From this differential inequality, the second moment will be zero in finite time and thus the solution blows up in finite time if,
\begin{equation*}
\F(u_0) < -\frac{C(M,C_1)}{2d(m-1)}. 
\end{equation*}
By Lemma \ref{lem:infF_Supercritical}, we may always find initial data with any given mass $M > 0$ such that this is true, since there exists infimizing sequences with vanishing second moments. The final assertion follows from Theorem \ref{thm:Continuation}.  Indeed, we have
\begin{align*}
T_\star \leq -\frac{I(0)}{2d(m-1)\F(u_0)+ C(M,C_1)}.  
\end{align*}

\end{proof}

\subsection{Critical Case: Theorems \ref{thm:SupercritMass} and \ref{thm:SupercritMass_2D}}
\noindent
The proof of Theorem \ref{thm:SupercritMass} follows the proof of Theorem \ref{thm:SupercritDiff}. 

\begin{lemma} \label{lem:infF_criticalMass}
Define $\mathcal{Y_M} = \set{u \in L^1 \cap L^\infty: u\geq 0, \norm{u}_1 = M}$. 
Suppose $\K$ satisfies \textbf{(B1)} and $A(u)$ satisfies \textbf{(B3)} for $m > 1$ and $\overline{A} > 0$.
Suppose further that the problem is critical, that is, $m = m^\star$ and let $M_c$ satisfy \eqref{cond:critical_mass_powerlaw}. 
If $M$ satisfies $M > M_c$, then $\inf_{\mathcal{Y_M}} \F = -\infty$. 
Moreover, there exists an infimizing sequence with vanishing second moments which converges to the Dirac delta mass in the sense of measures.
\end{lemma}
 
\begin{proof}
\noindent
We may proceed as in the proof of Lemma \ref{lem:infF_Supercritical}, but instead choose $\theta \in \left( (M_c/M)^{2-m^\star},1\right)$. 
Let $\alpha = d/p$. 
By optimality of $C_{m^\star}$, as before there exists
 $h^\star$ such that,
\begin{equation}
\theta C_{m^\star} \leq \frac{\abs{\int\int h^\star(x)h^\star(y)\abs{x-y}^{-\alpha}dx dy}}{\norm{h^\star}^{2-m^\star}_1\norm{h^\star}^{m^\star}_{m^\star}} \leq C_{m^\star}. \label{ineq:infF_theta}
\end{equation}
As above, we assume $h^\star \geq 0$ and $h^\star \in C_c^\infty$. \\ 

\noindent
Let $\mu = \norm{h^\star}_1^{1/d}M^{-1/d}$, $\lambda > 0$ and $h_\lambda(x) = \lambda^dh^\star(\lambda \mu x)$. By \textbf{(B1)} and \textbf{(B3)}, $\forall \, \epsilon > 0$ there exists a $\lambda$ and $R$ sufficiently large such that by $h^\star \in C_c^\infty$,
\begin{align*}
\F(h_\lambda)  \leq & \frac{\lambda^{dm - d}M}{(m^\star-1)\norm{h^\star}_1}(\overline{A}+\epsilon)\norm{h^\star}_{m^\star}^{m^\star} + C(R)\mu^{-d}\norm{h^\star}_1 \\
& - \frac{(\theta - \epsilon) C_{m^\star}}{2}\left(\frac{\norm{h^\star}_1}{M}\right)^{-2 + \alpha/d}\lambda^{\alpha} \norm{h^\star}_1^{2-m^\star}\norm{h^\star}_{m^\star}^{m^\star} 
\end{align*}
However in this case $\alpha = dm - d$ and  $m = m^\star$, therefore by \eqref{ineq:infF_theta} and Lemma \ref{lem:GHLS},
\begin{equation*}
\F(h_\lambda) \leq \lambda^{dm^\star - d}\norm{h^\star}_{m^\star}^{m^\star}\left[ \frac{M(\overline{A}+\epsilon)}{(m^\star-1)\norm{h^\star}_1} - \frac{(\theta - \epsilon) C_{m^\star}}{2}\left(\frac{\norm{h^\star}_1}{M}\right)^{-2 + \alpha/d}\norm{h^\star}_1^{2-m^\star} \right].
\end{equation*}
Then, 
\begin{equation*}
\F(h_\lambda) \leq \lambda^{dm^\star - d}\frac{\norm{h^\star}_{m^\star}^{m^\star}}{\norm{h^\star}_1}\left[ \frac{M(\overline{A}+\epsilon)}{(m^\star-1)} - \frac{(\theta - \epsilon)}{2}C_{m^\star} M^{2 - \alpha/d} \right].
\end{equation*}
Then since $\overline{A}/(m^\star-1) = C_{m^\star} M_c^{2-m^\star}/2$ and $\alpha/d - 1 = 2-m^\star$ we have,
\begin{equation*}
\F(h_\lambda) \leq \lambda^{dm^\star - d}\frac{\norm{h^\star}_m^m}{2\norm{h^\star}_1}C_{m^\star} M^{2-\alpha/d}\left[ \left(1 + \frac{\epsilon}{\overline{A}}\right)\left(\frac{M_c}{M}\right)^{2-m^\star} - (\theta - \epsilon) \right].
\end{equation*}
Since $\theta > (M_c/M)^{2-m^\star}$ we may take $\epsilon$ sufficiently small and $\lambda \rightarrow \infty$ to conclude that $\inf_{\mathcal{Y_M}} \F = -\infty$. 
As before, $h_\lambda$ converges to the Dirac delta mass in the sense of measures. 
\end{proof}
\noindent

\begin{proof} (Theorem \ref{thm:SupercritMass}) 
The theorem follows from a Virial identity as in Theorem \ref{thm:SupercritDiff}. 
\end{proof}

\begin{proof} (Theorem \ref{thm:SupercritMass_2D})
As in Theorem \ref{thm:SupercritDiff} we have by \textbf{(C2)}, \textbf{(C3)} and if $D$ is bounded, the convexity of the domain,
\begin{align*}
\frac{d}{dt}I(t) & \leq 2d\overline{A}\int A(u) dx + \int\int u(x) u(y) (x-y)\cdot \grad \K(x-y) dx dy  \\
& \leq 2dM\left( \overline{A} - \frac{cM}{2d} \right) + C_1M^{3/2}I^{1/2}. 
\end{align*}
Clearly, if $M > M_c$ then $I \rightarrow 0$ in finite time
if $I(0)$ is sufficiently small. 
\end{proof}

\section{Conclusion}
\noindent
The prior treatments of \eqref{AE} have restricted attention to either very singular kernels (PKS) 
or very smooth kernels (as in \cite{BertozziSlepcev10}). Moreover, most work has been restricted to power-law diffusion. 
We extend these approaches to develop a unified theory which applies to general nonlinear, degenerate diffusion and attractive kernels which are no more singular than the Newtonian potential. 
Existence arguments may apply to more singular kernels or unbounded initial data, however, to the authors' knowledge, 
Lemma \ref{lem:CZ} or something analogous must be available for any known uniqueness argument to hold.
We generalize the existing notions of criticality for PKS and show that the critical mass phenomenon observed in PKS is a generic property of critical aggregation diffusion models.
We extend the free energy methods of \cite{Dolbeault04,BlanchetEJDE06,Blanchet09} to obtain the sharp critical mass for a class of models
with general nonlinear diffusion and inhomogeneous kernels. In particular, we
show that the critical mass depends only on the singularity of the kernel 
at the origin and the growth of the entropy at infinity. The results presented here hold on bounded, convex domains for $d\geq 2$ and on $\Real^d$ for $d \geq 3$. 

\section{Appendix}

\subsection{Auxiliary Lemmas}

\begin{lemma} \label{L6}
Let $F$ be a convex $C^1$ function and $f=F'$.  Assume that $f(u)\in L^2(0,T,H^1(D)),\;u\in H^1(0,T,H^{-1}(D))$ and $F(u) \in L^\infty (0,T,L^1(D))$.  Then for almost all $0\leq s,\tau, \leq T$ the following holds:
\begin{align*}
\int \left(F(u(x,\tau))-F(u(x,s))\right)\;dx = \int_s^\tau \left\langle u_t, f(u(t))\right\rangle\;dt.   
\end{align*}
\end{lemma}

\begin{lemma}\label{L9}
Let $F(u,t)\in C^2([0,\infty),[0,\infty))$ be a convex function such that $F(0)=0$ and $F''>0$ on $(0,\infty)$.  Let $f_n$, for $n=1,2,...,$ and $f$ be a non-negative function on $D$ bounded from above by $M>0$.  Furthermore, assume that $f_n\rightharpoonup f\;\text{in}\;L^1(D)$ and $F(f_n) \rightarrow F(f)\; \text{in}\;L^1(D)$, then $\left\|f_n-f\right\|_{L^2(D)}\rightarrow 0$ as $n\rightarrow 0$.  
\end{lemma}

\begin{lemma} [Weak Lower-semicontinuity]\label{WLS}
Let $\rho_\epsilon$ be non-negative $L^1_{loc}(\pd)$ and $f_\epsilon$ a vector valued function in $L^1_{loc}(\pd)$ such that  $\forall \phi\in C_c^\infty(\overline{\pd}) and \xi\in C_c^\infty(\overline{\pd}, \Real^d)$
\begin{align*}
\intt \rho_\epsilon \phi dxdt &\rightarrow \intt \rho\phi dxdt \\
\intt f_\epsilon \cdot \xi dxdt &\rightarrow \intt f\cdot \xi dxdt.  
\end{align*}
\noindent Then
\begin{align*}
\intt \frac{1}{\rho} \abs{f}^2 dxdt \leq \liminf_{\epsilon\rightarrow 0} \intt \frac{1}{\rho_\epsilon} \abs{f_\epsilon}^2 dxdt 
\end{align*}
\end{lemma}

\subsection{Gagliardo-Nirenberg-Sobolev Inequality}
Gagliardo-Nirenberg-Sobolev inequalities are the main tool for obtaining $L^p$ estimates of PKS models and are used in many works, for instance \cite{Kowalczyk05,Blanchet09,SugiyamaADE07,JagerLuckhaus92}. The following inequality follows by interpolation and the classical Gagliardo-Nirenberg-Sobolev inequality.
\begin{lemma}[Inhomogeneous Gagliardo-Nirenberg-Sobolev] \label{lem:GNS}
Let $d \geq 2$ and $D \subset \Real^d$ satisfy the cone condition (see e.g. \cite{Adams03}). 
Let $f:D \rightarrow \Real$ satisfy $f \in L^p\cap L^q$ and $\grad f^k \in L^r$. Moreover let $1 \leq p \leq rk \leq dk$, $k < q < rkd/(d-r)$ and
\begin{equation}
\frac{1}{r} - \frac{k}{q} - \frac{s}{d} < 0. \label{cond:GNS}
\end{equation}
Then there exists a constant $C_{GNS}$ which depends on $s,p,q,r,d$ and the dimensions of the cone for which $D$ satisfies the cone condition such that 
\begin{equation}
\norm{f}_{L^q} \leq C_{GNS}\norm{f}^{\alpha_2}_{L^p} \norm{f^k}^{\alpha_1}_{W^{s,r}}, \label{eq:GNS}
\end{equation}
where $0 < \alpha_i$ satisfy
\begin{equation}
1 = \alpha_1 k + \alpha_2,
\end{equation}
and
\begin{equation}
\frac{1}{q} - \frac{1}{p} = \alpha_1(\frac{-s}{d} + \frac{1}{r} - \frac{k}{p}).
\end{equation}
\end{lemma}
\begin{proof}
We may assume that $f$ is Schwartz then argue by density. 
Let $\beta$ satisfy $\max(q,rk) < \beta < rkd/(d-r)$. First note by the Gagliardo-Nirenberg-Sobolev inequality, [Theorem 5.8, \cite{Adams03}], we have
\begin{align*}
\norm{f^k}_{\beta/k} & \lesssim_{\beta,k,r,s} \norm{f^k}^{1-\theta}_{r}\norm{f^k}^{\theta}_{W^{s,r}} \\
& \leq \norm{f^k}^{(1-\theta)(1-\mu)}_{p/k}\norm{f^k}^{(1-\theta)\mu}_{\beta/k}\norm{f^k}^{\theta}_{W^{s,r}},
\end{align*}
for $\mu \in (0,1)$ determined by interpolation and $\theta = s^{-1}(d/r - dk/\beta) \in (0,1)$.
Moreover, the implicit constant does not depend directly on the size of the domain. 
Therefore, 
\begin{equation*}
\norm{f^k}_{\beta/k} \lesssim \norm{f}_{p}^{(1-\theta)(1-\mu)/(1-\mu(1-\theta))}\norm{f^k}^{\theta/(1-\mu(1-\theta))}_{W^{s,r}}. 
\end{equation*}
Now, where $\lambda \in (0,1)$ determined by interpolation, 
\begin{align*}
\norm{f}_q & \leq \norm{f}_p^{(1-\lambda)}\norm{f^k}_{\beta/k}^{\lambda/k} \\
& \lesssim \norm{f}_{p}^{(1-\lambda) + (1-\theta)(1-\mu)/(1-\mu(1-\theta))}\norm{f^k}^{\lambda\theta/(k-k\mu(1-\theta))}_{W^{s,r}}.
\end{align*} 
\end{proof}

\subsection{Admissible Kernels} \label{Apx:Kernels}
\noindent
We now prove Lemmas \ref{lem:Newt},\ref{lem:CZ} and \ref{lem:CZ_HighReg}. We begin with the following characterizations of $\wkspace{p}$.
\begin{lemma} \label{lem:NotLwp}
Let $F(x) = f(\abs{x}) \in L_{loc}^1 \cap C^0\setminus\set{0}$ be monotone in a neighborhood of the origin. 
If $r^{-d/p} = o(f(r))$ as $r \rightarrow 0$, then $F \notin \wkspace{p}_{loc}$.
\end{lemma}
\begin{proof}
Since we have assumed $f$ to be monotone in a neighborhood of the origin, without loss of generality we prove the assertions assuming $f \geq 0$ on that neighborhood, since
corresponding work may be done if $f$ is negative. For any $\alpha > 0$, by monotonicity, we have a unique $r(\alpha)$ such that $f(r) > \alpha, \forall r < r(\alpha)$.
We thus have that $\lambda_f(\alpha) = \omega_dr(\alpha)^d$, where $\omega_d$ is the volume of the unit sphere in $\Real^d$. By the growth condition on $f$ and continuity we also have that for $\alpha$ sufficiently large,
\begin{equation}
\frac{1}{\epsilon} r(\alpha)^{-d/p} \leq f(r(\alpha)) = \alpha. \label{eq:dist_littleoh}
\end{equation}
Now, 
\begin{equation*}
\alpha^p\lambda_f(\alpha) = \omega_d\alpha^pr(\alpha)^d.
\end{equation*}
Hence, by (\ref{eq:dist_littleoh}) we have $\forall \,\epsilon > 0$ there is a neighborhood of infinity such that,
\begin{equation*}
\omega_d\alpha^pr(\alpha)^d \gtrsim \epsilon^{-p}.
\end{equation*}
We take $\epsilon \rightarrow 0$ to deduce that $F \notin \wkspace{p}$. 

\end{proof}

\begin{lemma} \label{lem:LwpChar}
Let $F(x) = f(\abs{x}) \in L_{loc}^1\cap C^0\setminus\set{0}$ be monotone in a neighborhood of the origin. 
Then $f \in \wkspace{p}_{loc}$ if and only if $f = \mathcal{O}(r^{-d/p})$ as $r \rightarrow 0$.  
\end{lemma}
\begin{proof}
Since we have assumed $f$ to be monotone in a neighborhood of the origin, without loss of generality we prove the assertions assuming $f \geq 0$ on that neighborhood. \\

\noindent
First assume that $f \neq \mathcal{O}(r^{-d/p})$ as $r \rightarrow 0$, which implies that for all $\delta_0 > 0$ and every $C > 0$ there exists an $r_C < \delta_0$ such that
\begin{equation*}
f(r_C) > Cr_C^{-d/p}. 
\end{equation*}
We now show that in a neighborhood of the origin, the function $f(r) - Cr^{-d/p}$ is strictly positive for $r < r_C$. 
Suppose not. Since both $f,r^{-d/p}$ are monotone, there exists $r_0$ such that $f(r) < Cr^{-d/p}$ for $r < r_0$.
However, this contradicts $f \neq \mathcal{O}(r^{-d/\gamma})$ as $r \rightarrow 0$. Thus, we have that
\begin{equation*}
f(r) > Cr^{-d/p}
\end{equation*}
in a neighborhood of the origin ($r < r_C$). Since for all $C > 0$ we can find a corresponding $r_C$, this is equivalent to $r^{-d/p} = o(f(r))$, and by Lemma \ref{lem:NotLwp} we have that $f \notin \wkspace{p}$. \\

\noindent
On the other hand, if $f = \mathcal{O}(r^{-d/p})$ as $r \rightarrow 0$ there exists $\delta > 0$ and $C > 0$ such that for all $r < \delta$,
\begin{equation}
f(r) \leq Cr^{-d/p}. \label{eq:fbigOh}
\end{equation}
By monotonicity, for all $\alpha > 0$ there is a unique $r(\alpha) \in [0,\delta]$ such that 
\begin{equation}
f(r) > \alpha, \textup{ for } r  < r(\alpha), \label{ineq:frbdd}
\end{equation}
where we take $r(\alpha) = 0$ if $f(r) < \alpha$ over the entire neighborhood. By \eqref{eq:fbigOh} and \eqref{ineq:frbdd}, we have,
necessarily that $r(\alpha) \lesssim \alpha^{-p/d}$. Therefore, 
\begin{equation*}
\alpha^p \lambda_f(\alpha) = \alpha^p\omega_d r(\alpha)^d \lesssim 1, 
\end{equation*}
which implies $f \mathbf{1}_{B_1(0)} \in \wkspace{p}$.  
\end{proof}

\begin{remark}
Similar statements may be made about the decay of $F(x)$ at infinity. 
\end{remark}

\begin{proof}(\textbf{Lemma \ref{lem:Newt}})
By the fundamental theorem of calculus and condition \textbf{(BD)}, 
\begin{align*}
\abs{\partial_{x_i}\partial_{x_j}\K(x)} & \leq \int_{1}^\infty \abs{\partial_r\partial_{x_i}\partial_{x_j}\K(rx)}dr \\
& \lesssim \abs{x}^{-d}. 
\end{align*}
Similarly, this argument also implies $\abs{\grad \K} \lesssim \abs{x}^{1-d}$, which in turn implies $\grad \K \in \wkspace{d/(d-1)}$. 
If $d > 2$ then we can carry out this argument another time and show that $\abs{\K} \lesssim \abs{x}^{2-d}$.
Moreover, in $d = 2$ we see that $\K$ could have, at worst, logarithmic singularities at zero and infinity. 

\end{proof}

\begin{proof}(\textbf{Lemma \ref{lem:CZ}})
We compute second derivatives of the kernel $\K$ in the sense of distributions.
Let $\phi \in C_c^\infty$, then by the dominated convergence theorem,
\begin{align*}
\int \partial_{x_i}\K \partial_{x_j} \phi dx & = \lim_{\epsilon \rightarrow 0} \int_{\abs{x} \geq \epsilon} \partial_{x_i}\K \partial_{x_j}\phi dx \\
& = -\lim_{\epsilon \rightarrow 0}\int_{\abs{x} = \epsilon} \partial_{x_j}\K(x)\frac{x_j}{\abs{x}}\phi(x) dS - \textup{PV} \int \partial_{x_i x_j}\K \phi dx.  
\end{align*}
By $\grad \K \in \wkspace{d/(d-1)}$ and Lemma \ref{lem:LwpChar}, we have $\grad \K = \mathcal{O}(\abs{x}^{1-d})$ as $x \rightarrow 0$. Therefore for $\epsilon$ sufficiently small, there exists $C > 0$ such that,
\begin{align*}
\abs{\int_{\abs{x} = \epsilon}\partial_{x_j}\K(x)\frac{x_j}{\abs{x}}\phi(x) dS} & \leq C{\int_{\abs{x} = \epsilon} \abs{x}^{1-d}\abs{\phi(x)}dS} \\
& = C\int_{\abs{x}= 1}\abs{\epsilon x}^{1-d}\abs{\phi(\epsilon x)} \epsilon^{d-1}dS = C\abs{\phi(0)}. 
\end{align*}
Similarly, we may define $\conv{D^2\K}{\phi}$ and we have,
\begin{equation*}
\norm{\conv{D^2\K}{\phi}}_p \leq C\norm{\phi}_p + \norm{\textup{PV}\int \partial_{x_i x_j}\K(y)\phi(x-y)dy}_p. 
\end{equation*}
Therefore, the first term can be extended to a bounded operator on $L^p$ for $1 \leq p \leq \infty$ by density.
The admissibility conditions \textbf{(R)},\textbf{(BD)} and \textbf{(KN)} are sufficient to apply the Calder\'{o}n-Zygmund inequality [Theorem 2.2 \cite{LittleStein}], 
which implies that the principal value integral in the second term is a bounded linear operator on $L^p$ for all $1 < p < \infty$. Moreover the proof provides an estimate of the operator norms,
\begin{equation*}
\norm{\textup{PV}\int \partial_{x_i,x_j}\K(y) u(x-y) dy}_p \lesssim \left\{
     \begin{array}{lr}
       \frac{1}{p-1}\norm{u}_p & 1 < p < 2 \\
       p \norm{u}_p & 2 \leq p < \infty.
     \end{array}
     \right. 
\end{equation*}

\end{proof}

\begin{proof}(\textbf{Lemma \ref{lem:CZ_HighReg}})
The assertion that $D^2\K \in \wkspace{\gamma}_{loc}$ implies $\K \in \wkspace{d/(d/\gamma - 2)}_{loc}$ follows similarly as in Lemma \ref{lem:Newt}. \\

\noindent
Now we prove the reverse implication. Let $\K \in \wkspace{d/(d/\gamma - 2)}_{loc}$. We show that $D^2\K = \mathcal{O}(r^{-d/\gamma})$ as $r \rightarrow 0$. 
Assume for contradiction that $D^2\K \neq \mathcal{O}(r^{-d/\gamma})$ as $r \rightarrow 0$.
This implies that $k^{\prime\prime} \neq \mathcal{O}(r^{-d/\gamma})$ or that $k^\prime(r)r^{-1} \neq \mathcal{O}(r^{-d/\gamma})$ as $r \rightarrow 0$. These two possibilities are essentially the same, so just 
assume that $k^{\prime\prime} \neq \mathcal{O}(r^{-d/\gamma})$. By monotonicity arguments used in the proof of Lemma \ref{lem:LwpChar}, this in turn implies
$r^{-d/\gamma} = o(k^{\prime\prime})$. However, this means that for all $\epsilon$, there exists a $\delta(\epsilon) > 0$ such that
for $r \in (0,\delta(\epsilon))$ we have,
\begin{align*}
k(r) - k(\delta(\epsilon)) = & -\int^r_{\delta(\epsilon)}k^\prime(s)ds = \int^r_{\delta(\epsilon)}\int^s_{\delta(\epsilon)}k^{\prime\prime}(t)dtds + (r-\delta(\epsilon))k^\prime(\delta(\epsilon)) \\
& \gtrsim \epsilon^{-1} r^{2-d/\gamma} + 1,
\end{align*}
which contradicts the fact that $k(r) = \mathcal{O}(r^{2-d/\gamma})$ as $r \rightarrow 0$ by Lemma \ref{lem:LwpChar}. \\

\noindent
The assertion regarding $\grad \K$ is proved in the same fashion. 
\end{proof}

\section*{Acknowledgments}
This work was supported in part by NSF Grant DMS-0907931 and NSF grant EFRI-1024765.  
The authors would like to thank Jos\'e A. Carrillo, Dong Li, Inwon Kim and Thomas Laurent for their helpful discussions.  We would also like to thank the referee for the careful reading of the paper and the insightful comments.

\vfill\eject
\bibliographystyle{plain}
\bibliography{nonlocal_eqns}

\begin{thebibliography}{10}

\bibitem{Adams03}
Robert~A. Adams and John J.~F. Fournier.
\newblock {\em Sobolev spaces}, volume 140 of {\em Pure and Applied Mathematics
  (Amsterdam)}.
\newblock Elsevier/Academic Press, Amsterdam, second edition, 2003.

\bibitem{Alikakos}
N.D. Alikakos.
\newblock {$L^p$} bounds of solutions to reaction-diffusion equations.
\newblock {\em Comm. Part. Diff. Eqn.}, 4:827--868, 1979.

\bibitem{AmbrosioGigliSavare}
L.A. Ambrosio, N.~Gigli, and G.~Sava\'re.
\newblock {\em Gradient flows in metric spaces and in the space of probability
  measures}.
\newblock Lectures in Mathematics, Birkh\"auser, 2005.

\bibitem{BertozziBrandman10}
A.L. Bertozzi and J.~Brandman.
\newblock Finite-time blow-up of {$L^\infty$}-weak solutions of an aggregation
  equation.
\newblock {\em Comm. Math. Sci.}, 8(1):45--65, 2010.

\bibitem{BertozziLaurent07}
A.L. Bertozzi and T.~Laurent.
\newblock Finite-time blow-up of solutions of an aggregation equation in
  {$\mathbb R^n$}.
\newblock {\em Comm. Math. Phys.}, 274:717--735, 2007.

\bibitem{BertozziLaurentRosado10}
A.L. Bertozzi, T.~Laurent, and J.~Rosado.
\newblock {$L^p$} theory for the multidimensional aggregation equation.
\newblock {\em Comm. Pure. Appl. Math.}, 64(1), 2010.

\bibitem{BertozziPugh98}
A.L. Bertozzi and M.C. Pugh.
\newblock Long-wave instabilities and saturation in thin film equations.
\newblock {\em Comm. Pure Appl. Math.}, 51(6):625--661, 1998.

\bibitem{BertozziSlepcev10}
A.L. Bertozzi and D.~Slep\v{c}ev.
\newblock Existence and uniqueness of solutions to an aggregation equation with
  degenerate diffusion.
\newblock {\em Comm. Pure. Appl. Anal.}, 9(6):1617--1637, 2010.

\bibitem{Biler09}
P.~Biler, G.~Karch, and P.~Lauren\c{c}ot.
\newblock Blowup of solutions to a diffusive aggregation model.
\newblock {\em Nonlinearity}, 22:1559--1568, 2009.

\bibitem{Biler06}
P.~Biler, G.~Karch, P.~Lauren\c{c}ot, and T.~Nadzieja.
\newblock The $8\pi$-problem for radially symmetric solutions of a chemotaxis
  model in the plane.
\newblock {\em Math. Meth. Appl. Sci}, 29:1563--1583, 2006.

\bibitem{BlanchetCalvezCarrillo08}
A.~Blanchet, V.~Calvez, and J.A. Carrillo.
\newblock Convergence of the mass-transport steepest descent scheme for
  subcritical {Patlak-Keller-Segel} model.
\newblock {\em SIAM J. Num. Anal.}, 46:691--721, 2008.

\bibitem{BlanchetCarlenCarrillo10}
A.~Blanchet, E.~Carlen, and J.A. Carrillo.
\newblock Functional inequalities, thick tails and asymptotics for the critical
  mass {Patlak-Keller-Segel} model.
\newblock {\em {arXiv}:1009.0134}, 2010.

\bibitem{Blanchet09}
A.~Blanchet, J.A. Carrillo, and P.~Lauren\c{c}ot.
\newblock Critical mass for a {Patlak-Keller-Segel} model with degenerate
  diffusion in higher dimensions.
\newblock {\em Calc. Var.}, 35:133--168, 2009.

\bibitem{Blanchet08}
A.~Blanchet, J.A. Carrillo, and N.~Masmoudi.
\newblock Infinite time aggregation for the critical {Patlak-Keller-Segel}
  model in {$\mathbb R^2$}.
\newblock {\em Comm. Pure Appl. Math.}, 61:1449--1481, 2008.

\bibitem{BlanchetEJDE06}
A.~Blanchet, J.~Dolbeault, and B.~Perthame.
\newblock Two-dimensional {Keller-Segel} model: Optimal critical mass and
  qualitative properties of the solutions.
\newblock {\em E. J. Diff. Eqn}, 2006(44):1--32, 2006.

\bibitem{Bio}
S.~Boi, V.~Capasso, and D.~Morale.
\newblock Modeling the aggregative behavior of ants of the species {\textit
  {p}olyergus rufescens}.
\newblock {\em Nonlinear Anal. Real World Appl.}, 1(1):163--176, 2000.
\newblock Spatial heterogeneity in ecological models (Alcal{\'a} de Henares,
  1998).

\bibitem{Burger07}
M.~Burger, V.~Capasso, and D.~Morale.
\newblock On an aggregation model with long and short range interactions.
\newblock {\em Nonlin. Anal. Real World Appl.}, 8(3):939--958, 2007.

\bibitem{CalvezCarrillo06}
V.~Calvez and J.A. Carrillo.
\newblock Volume effects in the {Keller-Segel} model: energy estimates
  preventing blow-up.
\newblock {\em J. Math. Pures Appl.}, 86:155--175, 2006.

\bibitem{Calvez}
V.~Calvez and L.~Corrias.
\newblock The parabolic-parabolic {K}eller-{S}egel model in {$\mathbb R^2$}.
\newblock {\em Commun. Math. Sci.}, 6(2):417--447, 2008.

\bibitem{CarlenLoss92}
E.~Carlen and M.~Loss.
\newblock Competing symmetries, the logarithmic {HLS} inequality and {Onofri's}
  inequality on {$\mathbb S^n$}.
\newblock {\em Geom. Func. Anal.}, 2(1):90--104, 1992.

\bibitem{Carrillo99}
J.~Carrillo.
\newblock Entropy solutions for nonlinear degenerate problems.
\newblock {\em Arch. Ration. Mech. Anal.}, 147(4):269--361, 1999.

\bibitem{CarrilloDiFrancesco09}
J.A. Carrillo, M.~DiFrancesco, A.~Figalli, T.~Laurent, and D.~Slep\v{c}ev.
\newblock Global-in-time weak measure solutions, finite-time aggregation and
  confinement for nonlocal interaction equations.
\newblock {\em Duke J. Math.}, 156(1):229--271, 2009.

\bibitem{CarrilloEntDiss01}
J.A. Carrillo, A.~J\"ungel, P.A. Markowich, G.~Toscani, and A.~Unterreiter.
\newblock Entropy dissipation methods for degenerate parabolic problems and
  generalized {Sobolev} inequalities.
\newblock {\em Montash. Math.}, 133:1--82, 2001.

\bibitem{CarrilloRosado10}
J.A. Carrillo and J.~Rosado.
\newblock Uniqueness of bounded solutions to aggregation equations by optimal
  transport methods.
\newblock {\em Proc. 5th Euro. Congress of Math. Amsterdam}, 2008.

\bibitem{Corrias04}
L.~Corrias, B.~Perthame, and H.~Zaag.
\newblock Global solutions of some chemotaxis and angiogenesis systems in high
  space dimensions.
\newblock {\em Milan J. Math.}, 72:1--28, 2004.

\bibitem{Rosado}
M.~Di~Francesco and J.~Rosado.
\newblock Fully parabolic {K}eller-{S}egel model for chemotaxis with prevention
  of overcrowding.
\newblock {\em Nonlinearity}, 21(11):2715--2730, 2008.

\bibitem{Dolbeault04}
J.~Dolbeault and B.~Perthame.
\newblock Optimal critical mass in the two dimensional {Keller-Segel} model in
  {$\mathbb R^2$}.
\newblock {\em C.R. Acad. Sci. Paris, S\'{e}r I Math}, 339(9):611--616, 2004.

\bibitem{Filbert}
Francis Filbet, Philippe Lauren{\c{c}}ot, and Beno{\^{\i}}t Perthame.
\newblock Derivation of hyperbolic models for chemosensitive movement.
\newblock {\em J. Math. Biol.}, 50(2):189--207, 2005.

\bibitem{GurtinMcCamy77}
E.~M. Gurtin and R.C McCamy.
\newblock On the diffusion of biological populations.
\newblock {\em Math. Biosci.}, 33:35--47, 1977.

\bibitem{Herrero2}
Miguel~A. Herrero and Juan J.~L. Vel{\'a}zquez.
\newblock Chemotactic collapse for the {K}eller-{S}egel model.
\newblock {\em J. Math. Biol.}, 35(2):177--194, 1996.

\bibitem{HandP2}
T.~Hillen and K.~Painter.
\newblock Global existence for a parabolic chemotaxis model with prevention of
  overcrowding.
\newblock {\em Adv. in Appl. Math.}, 26(4):280--301, 2001.

\bibitem{HandP}
T.~Hillen and K.~J. Painter.
\newblock A user's guide to {PDE} models for chemotaxis.
\newblock {\em J. Math. Biol.}, 58(1-2):183--217, 2009.

\bibitem{Hortsmann}
D.~Horstmann.
\newblock {From 1970 until present: the Keller-Segel model in chemotaxis and
  its consequences}.
\newblock {\em I, Jahresber. Deutsch. Math.-Verein}, 105(3):103--165, 2003.

\bibitem{JagerLuckhaus92}
W.~J\"ager and S.~Luckhaus.
\newblock On explosions of solutions to a system of partial differntial
  equations modelling chemotaxis.
\newblock {\em Trans. Amer. Math. Soc.}, 329(2):819--824, 1992.

\bibitem{KarchSuzuki10}
G.~Karch and K.~Suzuki.
\newblock Blow-up versus global existence of solutions to aggregation
  equations.
\newblock {\em arXiv:1004.4021v1}, 2010.

\bibitem{KS}
E.~F. Keller and L.A. Segel.
\newblock Model for chemotaxis.
\newblock {\em J. Theor. Biol.}, 30:225--234, 1971.

\bibitem{KillipVisanClay}
R.~Killip and M.~Vi\c{s}an.
\newblock Nonlinear {Schr\"odinger} equations at critical regularity.
\newblock {\em To appear in Proc. Clay summer school {``Evolution Equations"},
  June 23-July 18}, 2008.

\bibitem{Kowalczyk05}
R.~Kowalczyk.
\newblock Preventing blow-up in a chemotaxis model.
\newblock {\em J. Math. Anal. Appl.}, 305:566--588, 2005.

\bibitem{KowalczykSzymanska08}
R.~Kowalczyk and Z.~Szymanska.
\newblock On the global existence of solutions to an aggregation model.
\newblock {\em J. Math. Anal. Appl.}, 343:379--398, 2008.

\bibitem{Lapidus}
I.~Richard Lapidus and M.~Levandowsky.
\newblock Modeling chemosensory responses of swimming eukaryotes.
\newblock In {\em Biological growth and spread ({P}roc. {C}onf., {H}eidelberg,
  1979)}, volume~38 of {\em Lecture Notes in Biomath.}, pages 388--396.
  Springer, Berlin, 1980.

\bibitem{Laurent07}
T.~Laurent.
\newblock Local and global existence for an aggregation equation.
\newblock {\em Comm. Part. Diff. Eqn.}, 32:1941--1964, 2007.

\bibitem{LiRodrigo09}
D.~Li and J.~Rodrigo.
\newblock Finite-time singularities of an aggregation equation in
  {$\mathbb{R}^n$} with fractional dissipation.
\newblock {\em Comm. Math. Phys.}, 287:687--703, 2009.

\bibitem{LiRodrigoAM09}
D.~Li and J.~Rodrigo.
\newblock Refined blowup criteria and nonsymmetric blowup of an aggregation
  equation.
\newblock {\em Adv. Math.}, 220:1717--1738, 2009.

\bibitem{Lieb83}
E.H. Lieb.
\newblock Sharp constants in the {Hardy-Littlewood-Sobolev} and related
  inequalities.
\newblock {\em Ann. Math.}, 118:349--374, 1983.

\bibitem{Lieberman96}
Gary~M. Lieberman.
\newblock {\em Second order parabolic differential equations}.
\newblock World Scientific Publishing Co. Inc., River Edge, NJ, 1996.

\bibitem{LionsCC84}
P.L. Lions.
\newblock The concentration-compactness principle in the calculus of
  variations. the locally compact case, part 1.
\newblock {\em Ann. I.H.P., Anal. Nonlin.}, 1(2):109--145, 1984.

\bibitem{LoeperSG06}
G.~Loeper.
\newblock A fully nonlinear version of the incompressible {Euler} equations:
  the semigeostrophic system.
\newblock {\em SIAM J. Math. Anal.}, 38(3):795--823, 2006.

\bibitem{LoeperVP06}
G.~Loeper.
\newblock Uniqueness of the solution to the {Vlasov-Poisson} system with
  bounded density.
\newblock {\em J. Math. Pures Appl.}, 86:68--79, 2006.

\bibitem{MajdaBertozzi}
A.~Majda and A.~L. Bertozzi.
\newblock {\em Vorticity and Incompressible Flow}.
\newblock Cambridge University Press, 2002.

\bibitem{McCann97}
R.J. McCann.
\newblock A convexity principle for interacting gases.
\newblock {\em Adv. Math.}, 128:153--179, 1997.

\bibitem{Milewski}
P.~A. Milewski and X.~Yang.
\newblock A simple model for biological aggregation with asymmetric sensing.
\newblock {\em Comm. Math. Sci.}, 6(2):397--416, 2008.

\bibitem{Otto01}
F.~Otto.
\newblock The geometry of dissipative evolution equations: the porous medium
  equation.
\newblock {\em Comm. Part. Diff. Eqn.}, 26(1):101--174, 2001.

\bibitem{Patlak}
C.~S. Patlak.
\newblock Random walk with persistence and external bias.
\newblock {\em Bull. Math. Biophys.}, 15:311--338, 1953.

\bibitem{Robert97}
R.~Robert.
\newblock Unicite\'e de la solution faible \`a support compact de l'\'equation
  de {Vlasov-Poisson}.
\newblock {\em C.R. Acad. Sci. Paris, S\'{e}r. I Math}, 324(8):873--877, 1997.

\bibitem{Slepcev08}
D.~Slep\v{c}ev.
\newblock Coarsening in nonlocal interfacial systems.
\newblock {\em SIAM J. Math. Anal.}, 40(3):1029--1048, 2008.

\bibitem{LittleStein}
E.~Stein.
\newblock {\em Singualar Integrals and Differentiability Properties of
  Functions}.
\newblock Princeton University Press, 1970.

\bibitem{SugiyamaDIE06}
Y.~Sugiyama.
\newblock Global existence in sub-critical cases and finite time blow-up in
  super-critical cases to degenerate {Keller-Segel} systems.
\newblock {\em Diff. Int. Eqns.}, 19(8):841--876, 2006.

\bibitem{SugiyamaADE07}
Y.~Sugiyama.
\newblock Application of the best constant of the {Sobolev} inequality to
  degenerate {Keller-Segel} models.
\newblock {\em Adv. Diff. Eqns.}, 12(2):121--144, 2007.

\bibitem{SugiyamaDIE07}
Y.~Sugiyama.
\newblock The global existence and asymptotic behavior of solutions to
  degenerate to quasi-linear parabolic systems of chemotaxis.
\newblock {\em Diff. Int. Eqns.}, 20(2):133--180, 2007.

\bibitem{Topaz}
C.~M. Topaz, A.~L. Bertozzi, and M.~A. Lewis.
\newblock A nonlocal continuum model for biological aggregation.
\newblock {\em Bull. Math. Biol.}, 68(7):1601--1623, 2006.

\bibitem{VazquezPME}
J.L. V\'{a}zquez.
\newblock {\em The Porous Medium Equations}.
\newblock Clarendon Press, Oxford, 2007.

\bibitem{Weinstein83}
M.~Weinstein.
\newblock Nonlinear {Schr\"odinger} equations and sharp interpolation
  estimates.
\newblock {\em Comm. Math. Phys.}, 87:567--576, 1983.

\bibitem{Witelski04}
T.P. Witelski, A.J. Bernoff, and A.L. Bertozzi.
\newblock Blowup and dissipation in a critical-case unstable thin film
  equation.
\newblock {\em Euro. J. Appl. Math.}, 15:223--256, 2004.

\bibitem{Yudovich63}
V.I. Yudovich.
\newblock Non-stationary flow of an ideal incompressible liquid.
\newblock {\em Zh. Vychisl. Mat. Fiz.}, 3(6):1032--1066, 1963.

\end{thebibliography}

\end{document}